\documentclass[reqno]{amsart}
\usepackage[
    bookmarksopen=true,
    citecolor=blue]{hyperref}
\usepackage{multirow,makecell}
\usepackage{float}
\usepackage{caption}
\usepackage[all]{xy}
\usepackage{eufrak}
\usepackage{mathrsfs, amscd}
\usepackage{amsmath}
\usepackage{amssymb}

\begin{document}
\newtheorem{theorem}{Theorem}[section]
\newtheorem{claim}{Claim}
\newtheorem{notation}{Notation}[section]
\newtheorem{lemma}{Lemma}[section]
\newtheorem{prop}{Proposition}[section]
\newtheorem{cor}{Corollary}[section]
\newtheorem{definition}{Definition}[section]
\newtheorem{example}{Example}[section]
\newtheorem{remark}{Remark}[section]
\newtheorem{conjecture}{Conjecture}[section]

\newcommand{\g}{\frak g}
\newcommand{\field}{\mathbb{F}}
\newcommand{\gl}{\mathfrak{gl}}
\newcommand{\der}{{\rm der}}
\newcommand{\inder}{{\rm inder}}
\newcommand{\coder}{{\rm coder}}
\newcommand{\D}{\partial}
\newcommand{\Alt}{{\rm Alt}}
\newcommand{\HD}{\mathscr{D}}
\newcommand{\G}{\mathfrak{G}}
\newcommand{\Dn}{\mathfrak{D}}
\newcommand{\incoder}{{\rm incoder}}
\newcommand{\Hom}{{\rm Hom}}
\newcommand{\Aut}{{\rm Aut}}
\newcommand{\Int}{{\rm Int}}
\newcommand{\End}{{\rm End}}
\newcommand{\sgn}{{\rm sgn}}
\newcommand{\Hr}{{\rm H}}
\newcommand{\Zr}{{\rm Z}}
\newcommand{\Br}{{\rm B}}
\newcommand{\F}{\mathbb{F}}
\newcommand{\h}{\frak h}
\newcommand{\he}{{\rm h}_\varepsilon}
\newcommand{\Z}{\mathcal{Z}}
\newcommand{\B}{\mathcal{B}}
\newcommand{\xx}{{\rm x}}
\newcommand{\y}{{\rm y}}
\newcommand{\z}{{\rm z}}
\newcommand{\im}{{\rm Im}}
\newcommand{\Id}{\rm Id}
\newcommand{\Ob}{\rm Ob}
\newcommand{\rC}{\rm Coder}

\let \d=\Delta
\let \t=\otimes
\let \c=\circ
\let \dis =\displaystyle
\let \bb=\mathbb{}
\let \g = \gamma
\let \o = \omega
\let \O = \Omega
\let \ep  = \epsilon
\let \v = \vskip
\let \n = \newline
\let \p = \newpage
\let \cl = \centerline
\let \nd = \noindent

\title[Cohomologies and deformations of coassociative coderivations]{Cohomologies and deformations of coassociative coderivations}

%    Remove any unused author tags.

%    author one information
\author[Du]{Lei Du}
\address{School of Mathematical Sciences, Anhui University, Hefei, 230601, China}
\curraddr{}
\email{18715070845@163.com}
%    author two information
\author[Ma]{Yashuang Ma}
\address{School of Mathematical Sciences, Anhui University, Hefei, 230601, China}
\curraddr{}
\email{mayashuangg@163.com}
%    author three information
\author[Xv]{Jiangnan Xv}
\address{School of Mathematical Sciences, Anhui University, Hefei, 230601, China}
\curraddr{}
\email{jnxu1@outlook.com}
%    author one information
\author[Bao]{Yanhong Bao}
\address{School of Mathematical Sciences, Anhui University, Hefei, 230601, China}
\curraddr{}
\email{baoyh@ahu.edu.cn}

\thanks{L. Du is the corresponding author. }
\thanks{Y.-H. Bao was supported by NSFC (Grant No. 11871071)}



\date{}

\dedicatory{}

\begin{abstract}
The motivation of this paper is to construct a deformation theory of  coderivations of coassociative coalgebras. We introduce  a notion of a Coder pair, that is, a coassociative coalgebra with a coderivation. Then we define a proper corepresentation of a Coder pair and study the corresponding cohomology. Finally, we show that a Coder pair is rigid, provided that the second cohomology group is trivial and point out that  a deformation of finite order is   extensible to higher order deformations if the obstruction class, which is defined to be in the third cohomology group, is trivial.
\end{abstract}

\maketitle

\section{Introduction}
Among a classical approach to study a mathematical structure via invariants, cohomology theories occupy a significant position as they enable to control deformations problems. Cohomology theories  have been developed with a great success on associative algebras\cite{GS0,GH}, Lie algebras \cite{CE}, group algebras \cite{SS} and conformal algebras\cite{BK,KK}. Algebraic deformation theory is firstly described due to the seminal work of Gerstenhaber \cite{GS1}. After that, the cohomology and obstruction theory of various kinds algebras has been studied widely. For example, Nijenhuis and
Richardson \cite{lie} describe deformations of graded Lie algebras and Balavoine \cite{BA}  studies deformations of any algebra over a quadratic operad.

 Gerstenhaber and Schack \cite{GS,GS4} naturally extend a deformation theory to algebraic
morphisms and, more generally, diagrams. In  \cite{DY2}, one constructs a deformation theory of a coalgebra morphism. Recently,  deformations of algebraic  derivations are devoloped \cite{Das,Tang}. Deformations of derivations are more difficult to describe than that of the algebraic objects themselves. In \cite{Tang}, the Nijenhuis-Richardson bracket is used to prove that the obstruction classes of derivatons of  Lie algebras are cocycles. In the similar direction, Das and Mandal \cite{Das} obtain the similar result in the associative case by using the Gerstenhaber bracket.

The notion of coassociative coalgebras has received the attention of mathematicians as a part of the study of Hopf algebras. Coalgebras are not only the dual to algebras, but they also have many special properties. For example, unlike the case of algebras, ``The Fundamental Theorem on Coalgebras" \cite[Chapter 2]{SW} asserts that every element of a coassociative coalgebra is contained in a finite dimensional subcoalgebra.

On the other hand, coalgebras and their coderivations are useful. In \cite{DE}, the author shows that there exists  the canonical (anti-homomorphic) embedding of $L$ into the Lie algebra of coderivations of the universal enveloping algebra $U(L)$ of $L$ if $L$ be a Lie algebra over a field  of characteristic $p>0$. One \cite{ST} can mention that the Hochschild cohomology of a algebra $A$ is isomorphic to the space of coderivations on the tensor  coalgebra of $A$. In \cite{NY}, one can also show that the space of outer $A_\infty$-coderivations of $P$  is isomorphic to the Hochschild cohomology of $A$, where $P$ is a bimodule resolution of $A$.

These above facts motivate us to construct a cohomology theory of coderivations of coassociative coalgebras that controls deformations of  coderivations. There are two sides of deformations of coderivations  that are different than the algebraic case. First, by reason of having no  a coalgebra version of the Gerstenhaber bracket or the Nijenhuis-Richardson bracket, it differs from the algebraic case to deal with the obstruction class  of coderivations, using instead an elementary, computational approach. Second, as what one described in \cite{DY1}, working with coalgebras seems to be even more
conceptual and transparent than in the algebraic case. Our coderivation deformation theory is element-free, i.e.,  the Hochschild coalgebra cochain is the form $\Hom(M,C^{\otimes n})$, where $C$ is a coalgebra and $M$ is a bicomodule of $C$, but it is never needed to pick elements in $M$ in the  process of our coderivation deformation.

For convenience to describe,  the main object of study of this paper is the notion of a Coder pair, that is, a coassociative coalgebra  equipped with a coderivation. The paper is organized as follows.

In Section \ref{secLiepair}, we define the concept of  a Coder pair, which is a coalgebra with a coderivation and consider the dual of Coder pairs by relating the  category of Coder pairs with  the category of Der pairs. In Section \ref{seccomplex}, we give a proper corepresentation of a Coder pair and construct a complex associated to a Coder pair with coefficients in a corepresentation over it, then prove it is indeed a complex. In Section \ref{secdefomation}, we show that infinitesimal deformations are a 2-cocycles in the complex  of a Coder pair with coefficients in the coadjoint bicomodule (Proposition \ref{z2}), and we deduce that a Coder pair is rigid, provided that the second cohomology group is trivial (Theorem \ref{Thm-rigid}). Finally, the obstruction that extends a 2-cocycle to a deformation of a Coder pair is identified, which is shown to be a $3$-cocycle (Proposition \ref{ext2}). Further, we show that extensions to higher order deformations automatically exist if the corresponding third cohomology group is trivial (Theorem \ref{ext1}).

\vspace{3mm}
In this paper, we work over a field $\field$ of characteristic 0 and unless otherwise specified, linear spaces, linear maps, $\otimes$, $\Hom$, $\End$ are defined over $\field$. We shall sometimes  use Sweedler's notation.

\section{Coder pairs and their duality}\label{secLiepair}
\subsection{Coder piars}
In this subsection, we consider a pair of a coassociative coalgebra $C$ together with a coderivation $\psi_C$. Such a pair $(C, \psi_C)$ of a coassociative coalgebra with a coderivation is called a Coder pair.

A coassociative coalgebra over $\mathbb{F}$ is a pair $(C, \Delta)$  in which $C$ is a vector space over $\mathbb{F}$ and $\Delta \colon C \to C\otimes C$ is a linear map such that (See \cite[Chapter 1]{SW})
\begin{eqnarray}\label{defcoalgebras}
({{\Id}}_C \otimes \Delta) \Delta = (\Delta \otimes {{\Id}}_C) \Delta.
\end{eqnarray}
\\
\indent A $C$-bicomodule\cite{BW} is a vector space $M$ together with a left action map $\rho_ l : M \to C \otimes M$, and a right action map $\rho_r \colon M \to M \otimes C$, such that $({\d} \otimes {{\Id}}_M) \circ \rho_l = ({{\Id}}_C \otimes \rho_l)\circ  \rho_l$, $({{\Id}}_M \otimes \d) \circ \rho_r = (\rho_r \otimes {\Id}_C)\circ  \rho_r$, and $({{\Id}}_C \t \rho_r) \circ \rho_l=(\rho_l \otimes {{\Id}}_C)\circ  \rho_r$.

If $\psi_C\in\End(C)$ satisfies that
\begin{eqnarray}\label{eq-coderiCM}
\Delta \circ \psi_C = (\psi_C\otimes {{\Id}}_C)\circ \Delta + ({{\Id}}_C\otimes \psi_C)\circ \Delta: \ C\rightarrow C\otimes C,
\end{eqnarray}
then $\psi_C$ is a coderivation of $C$. Denote by $\coder(C)$ the set of all coderivations of $C$.

\begin{definition}
A Coder pair is a coassociative coalgebra $C$ with a coderivation $\psi_C$, which is denoted by $(C, \psi_C)$.
\end{definition}

\begin{definition}
Let $(C, \psi_C)$ and $(D, \psi_D)$ be Coder pairs. A morphism $f$ from $(C, \psi_C)$ to $(D, \psi_D)$ is a coalgebra morphism $f:C\rightarrow D$ such that $\psi_D\circ f=f\circ\psi_C$.
\end{definition}
The category of Coder pairs and their morphisms is denoted by $\mathbf{Coder}$.
\begin{example}
 Let $S$ be the set $\{v_m \mid m\in \mathbf{N}\}$ and $V$ be a $\field$-vector  space with basis $S$.  Then  $V$ is a coalgebra with  comultiplication $\Delta$ defined by
\begin{align*}
\Delta(v_m)=\sum_{i=0}^{m} {m\choose i} v_{m-i}\otimes v_{i}, m\in \mathbf{N}.
\end{align*}
The coalgebra is called the divided power coalgebra in \cite[Chapter 1]{DNR}.

Let $\partial: V\rightarrow V$ be a linear map defined by $\partial(v_m)= v_{m+1}$ for any  $v_m\in S$. We have
\begin{align*}
&((\partial\otimes {\Id}_C)\circ\Delta)(v_m)+(({\Id}_C\otimes \partial)\circ\Delta)(v_m)\\
=&(\partial\otimes {\Id}_C)(\sum_{i=0}^{m}{m\choose i}v_{m-i}\otimes v_i)+({\Id}_C\otimes \partial)(\sum_{i=0}^{m}{m\choose i}v_{m-i}\otimes v_i)\\
=&\sum_{i=0}^{m}{m\choose i} v_{m+1-i}\otimes v_i+\sum_{i=0}^{m}{m\choose i} v_{m-i}\otimes v_{i+1}\\
=&\sum_{i=0}^{m+1}{m+1\choose i}v_{m+1-i}\otimes v_{i}\\
=&\Delta(v_{m+1})=(\Delta\circ\partial)(v_m)
\end{align*}
which follows that $\partial$ is a coderivation of $V$, that is, $(V, \partial)$ is a Coder pair. $\Box$
\end{example}

\begin{example}
Let M be a vector space over $\field$. The tensor algebra $T(M)=\oplus_{n\geqslant 0}T^{n}(M)$ has a bialgebra structure, whose coalgebra structure $\Delta$ is defined by  {\rm (See \cite[Chapter 3]{SW})}
\begin{align*}
\Delta(m)=m\otimes 1 +1\otimes m, m\in M.
\end{align*}
 $(T(M),\Delta)$ is a coalgebra whose underlying vector space is the same as that of the tensor algebra.

Let  $r_m \in\End(T(M))$ be a right translation by $m$, that is, $r_m : x\mapsto xm, \text{ for any } x\in T(M)$. Then $(T(M), r_m)$ is a Coder pair. In fact, for any $x\in T(M)$, we have
\begin{align*}
(\Delta\circ r_m) (x)=& \Delta(xm)\\
=&\Delta(x)\Delta(m) \ \ \  (\text{since $T(M)$ is a bialgebra})\\
=&(\Sigma x_{(1)}\otimes x_{(2)})(m\otimes 1+1\otimes m)  \ \ \  (\text{Sweedler's notation})\\
=&\Sigma x_{(1)}m\otimes x_{(2)}+\Sigma x_{(1)}\otimes x_{(2)}m\\
=&((r_m\otimes  1)\circ\Delta)(x)+ ((1\otimes r_m)\circ\Delta)(x),
\end{align*}
which means that $r_m$ is a coderivation.   \ \ \ \ \ \ \ \ \ \ \ \ \ \ \ \ \ \  \  \ \ \  $\Box$
\end{example}

\subsection{Duality of Coder pairs}
Let $(A, u)$ be an associative algebra where $u : A\otimes A\rightarrow A$ is  the multiplication of $A$, i.e., $u(a\otimes b)= a\cdot b$ for $a, b\in A$. A linear map $\varphi_A : A\rightarrow A$ is a derivation if $\varphi_A$ satisfies
\begin{align}\label{eq-derivation}
\varphi_A\circ u=u\circ (\varphi_A \otimes {\Id})+u\circ ({\Id}\otimes \varphi_A),
\end{align}
that is,
\begin{align}\label{eq-derivation}
\varphi_A(a\cdot b)=\varphi_A(a)\cdot b+a\cdot \varphi_A(b), \ \  a,b\in A.
\end{align}

Let $(A, \varphi_A)$ be a Der pair, that is, an associative algebra $A$ with a derivation $\varphi_A$ (See \cite{Das} \cite{Tang}). Then a morphism
$f : (A, \varphi_A) \rightarrow (B, \varphi_B)$ is an algebra morphism $f : A\rightarrow B$ such that $\varphi_B\circ f=f\circ\varphi_A$.
We note the category of Der pairs by {\bf Der} and consider the dual of Coder pairs by relating the  category of Coder pairs with  the category of Der pairs.  To this, we need the lemmas as follows.

\begin{lemma}\label{derivation1}
If $(C,\psi_C)$ is a Coder pair, then $(C^{*},\psi_C^*)$ is a Der pair.
\end{lemma}
\begin{proof}
Suppose that $(C^*, \cdot)$ is an algebra when $(C, \Delta)$ is a coalgebra(See \cite[Chapter 1]{SW}). We only show that $\psi_C^*$ is a derivation when $\psi_C$ is a coderivation. Fix any $f, g\in C^\ast$ and $c \in C$, we have
\begin{eqnarray*}
\psi_C^{*}(f\cdot g)(c)&=& (f\cdot g)(\psi_C(c))
\ = \  \sum f(\psi_C(c)_{(1)})g(\psi_C(c)_{(2)})
\\
&=&(f\otimes g)\left(\sum \psi_C(c)_{(1)}\otimes \psi_C(c)_{(2)}\right)\ =\ (f\otimes g)((\Delta \circ \psi_C)(c))\\
&\stackrel{(\ref{eq-coderiCM})}{=}&(f\otimes g)((\psi_C\otimes {\Id}_C)\circ \Delta + ({\Id}_C\otimes \psi_C)\circ \Delta)(c)
\\
&=&(f\otimes g)\left(\sum \psi_C(c_{(1)})\otimes c_{(2)}+ \sum c_{(1)}\otimes \psi_C(c_{(2)})\right)
\\
&=&\sum f(\psi_C(c_{(1)})g(c_{(2)}) +\sum f(c_{(1)})g(\psi_C(c_{(2)}))
\\
&=&(\psi_C^*(f)\otimes g)\Delta(c)+(f\otimes \psi_C^*(g))\Delta(c)\\
&=&(\psi_C^{*}(f)\cdot g)(c)+(f\cdot \psi_C^{*}(g))(c),
\end{eqnarray*}
which means that $\psi_C^{*}(f\cdot g)=\psi_C^{*}(f)\cdot g+f\cdot\psi_C^{*}(g)$ as required, and hence $\psi_C^\ast$ is a derivation.
\end{proof}
Let $A$ be an associative algebra, the zero dual coalgebra
\begin{align}\label{eq-A0}
A^0 =\{f\in A^*\mid\ker(f)  \text{\ \ contains an ideal of finite codimension}\}
\end{align}
of an algebra is defined in \cite[Chapter 1]{SW}. Compared to obtaining a Der pair from  a Coder pair, it is completely different to get a Coder pair from a Der pair as the following lemma.
\begin{lemma}\label{derivation2}
If $(A, \varphi_A)$ is a Der pair, then $\varphi_A^*(A^0)\subseteq A^0$ and the induced map $\varphi_A^0 : A^0\rightarrow A^0$ is a coderivation; i.e., $(A^0, \varphi_A^0)$ is a Coder pair.
\end{lemma}
\begin{proof}
First, we prove that $\varphi_A^*(A^0)\subseteq A^0$.
Let $f\in A^0$ and $I$ be an ideal of finite codimension in $A$ which is contained in $\ker f$. Assume that  $\varphi_A^{-1}(I)=\{a\in A\mid \varphi_A(a)\in I\}\subseteq \ker \varphi_A^*(f)$.  We claim that $I\cap \varphi^{-1}_A(I)$ is an  ideal of $A$. In fact, let $x\in I\cap \varphi_A^{-1}(I)$ and $a\in A$, we have
$ax\in I$ and $\varphi_A(ax)=\varphi(a)x+a\varphi_A(x)\in I$, which follow that $ax \in I\cap \varphi_A^{-1}(I)$. Similarly, $xa\in I\cap \varphi_A^{-1}(I)$. It implies that $I\cap \varphi^{-1}_A(I)$ is an  ideal of $A$.
Then we claim that $\dim A/(I\cap \varphi_A^{-1}(I))< \infty$.  In fact,
$\varphi_A^{-1}(I)=\ker (\xymatrix@C=0.5cm{
  A \ar[r]^{\varphi_A} & A \ar[r]^{\pi} &A/I })$,
so $\dim A/\varphi_A^{-1}(I)=\dim \im (\pi\circ \varphi_A)\leqslant \dim A/I <\infty$, which follows that  $\dim A/(I\cap \varphi_A^{-1}(I))< \infty$.
Hence $I\cap \varphi_A^{-1}(I) $ is an ideal of finite codimension in $A$, which is contained in $\ker \varphi_A^*(f)$,  that is, $\varphi_A^*(f) \in A^0$.

Second, we show that $\varphi_A^0 : A^0\rightarrow A^0$ is a coderivation. Let $u:A\otimes A\rightarrow A$ be the multiplication of $A$, then $\Delta :  A^0\rightarrow A^0 \otimes A^0, \Delta =\iota^{-1} \circ u^{\ast}$ is the comultiplication of $A^0$, where $\iota : A^\ast \otimes A^\ast \rightarrow (A\otimes A)^\ast$ is the canonical injiection. Note that $\Delta(a^*)=\Sigma a_{(1)}^*\otimes a_{(2)}^* \text{ for } a^*\in A^0$, then
\begin{align*}
&((\varphi_A^0\otimes {\Id})\circ\d)(a^*)+(({\Id} \otimes \varphi_A^0)\circ\d)(a^*)\\
=&(\varphi_A^0\otimes {\Id})(\Sigma a_{(1)}^*\otimes a_{(2)}^*)+({\Id} \otimes \varphi_A^0)(\Sigma a_{(1)}^*\otimes a_{(2)}^*)\\
=&\Sigma \varphi_A^0(a_{(1)}^*)\otimes a_{(2)}^*+  \Sigma a_{(1)}^*\otimes \varphi_A^0(a_{(2)}^*)\\
=&\Sigma a_{(1)}^*\circ \varphi_A\otimes a_{(2)}^*+  \Sigma a_{(1)}^*\otimes a_{(2)}^*\circ \varphi_A\\
=&\d(a^*)\circ (\varphi_A\otimes {\Id})+ \d(a^*)\circ ({\Id}\otimes \varphi_A)\\
=&(\iota^{-1}\circ u^*)(a^*)\circ (\varphi_A\otimes {\Id})+ (\iota^{-1}\circ u^*)(a^*)\circ ({\Id}\otimes \varphi_A)\\
=&\iota^{-1}(a^*\circ u\circ (\varphi_A\otimes {\Id}))+ \iota^{-1}(a^*\circ u\circ ({\Id}\otimes \varphi_A))\\
\stackrel{(\ref{eq-derivation})}{=}&\iota^{-1}(a^*\circ \varphi_A\circ u)\\
=&(\iota^{-1}\circ u^*\circ\varphi_A^0)(a^*) \\
=&((\d\circ\varphi^0_A))(a^*),
\end{align*}
which follows that $\varphi^0_A$ is a coderivation.
\end{proof}
Combining Lemma \ref{derivation1} and Lemma \ref{derivation2}, we get the following result.
\begin{theorem}
Let $(C, \psi_C)$ be a Coder pair and $(A, \varphi_A)$ be a Der pair, then there is an adjointness result between the functors $(A, \varphi_A)\rightarrow (A^0, \varphi_A^0)$ and $(C, \psi_C)\rightarrow (C^*,\psi_C^*)$, i.e.,
$$\Hom_{\bf Coder}((C, \psi_C), (A^0, \varphi_A^0))\cong\Hom_{\bf Der}((A, \varphi_A), (C^*, \psi_C^*)).$$
\end{theorem}
\begin{proof}
By Lemma \ref{derivation1} and Lemma \ref{derivation2}, $(A^0, \varphi_A^0)$ is a Coder pair and $(C^*,\psi_C^*)$ is a Der pair. Then we define the maps
\begin{align*}
&F : \Hom_{\bf Coder}((C, \psi_C), (A^0, \varphi_A^0))\rightarrow \Hom_{\bf Der}((A, \varphi_A), (C^*, \psi_C^*))\\
&G : \Hom_{\bf Der}((A, \varphi_A), (C^*, \psi_C^*))\rightarrow \Hom_{\bf Coder}((C, \psi_C), (A^0, \varphi_A^0))
\end{align*}
by $F(f)=f^*\circ i_A$ for $f : (C, \psi_C)\rightarrow (A^0, \varphi_A^0)$ a morphism of Coder piars and $G(g)=g^0\circ\phi_C$ for
$g : (A, \varphi_A)\rightarrow (C^*, \psi_C^*)$ a morphism of Der piars,  where $i_A : A\rightarrow A^{0*}$ is defined by $i_A(a)(a^*)=a^*(a) \text{ for } a\in A, a^*\in A^0$ and $\phi_C : C\rightarrow C^{*0}$ is defined by $\phi_C(c)(c^*)=c^*(c) \text{ for } c\in C, c^*\in C^0$.

First, we show that $F(f)$ is a morphism of Der piars.
Since $f$ is a morphism of Coder piars, we have $\varphi_A\circ f=f\circ \psi_C$. Then
\begin{align*}
(\psi_C^*\circ f^*\circ i_A)(a)(c)=&i_A(a)((f\circ\psi_C)(c))\\
=&i_A(a)((\psi_A^0\circ f)(c))\\
=&((\varphi_A^0\circ f)(c))(a)\\
=&(f(c))(\varphi_A(a))\\
=&i_A(\varphi_A(a))(f(c))\\
=&(f^*\circ i_A\circ \varphi_A)(a)(c),
\end{align*}
which implies $F(f)=f^*\circ i_A$ is a morphism of Coder piars. Similarly,  $G(g)=g^0\circ\phi_C$ is a morphism of Coder piars.

Second,  we prove that the maps $F$ and $G$ are inverse one to each other. In fact, for $f\in \Hom_{\bf Coder}((C, \psi_C), (A^0, \varphi_A^0)), c\in C$ and $a\in A$, we have
\begin{align*}
&((G\circ F)(f))(c)(a)\\
=&((f^*\circ i_A)^0\circ\phi_C)(c)(a)\\
=&(f^*\circ i_A)^0(\phi_C(c))(a)\\
=&\phi_C(c)((f^*\circ i_A)(a))\\
=&(f^*\circ i_A)(a)(c)\\
=&i_A(a)(f(c))\\
=&f(c)(a)
\end{align*}
which implies $G\circ F= \Id$. Similarly, we have $F\circ G= \Id$.

Hence, we have $$\Hom_{\bf Coder}((C, \psi_C), (A^0, \varphi_A^0))\cong\Hom_{\bf Der}((A, \varphi_A), (C^*, \psi_C^*)).$$
\end{proof}
\begin{remark}
If we restrict these two functors to the finite dimension case, we obtain the duality of categories described as $$\xymatrix@C=0.5cm{(A,\varphi_A)\ar[r]^{\cong}&(A^{0*}, \varphi_A^{0*})}, \xymatrix@C=0.5cm{(C, \psi_C)\ar[r]^{\cong}&(C^{*0}, \psi_C^{*0})}.$$
\end{remark}

\section{Cohomology complexs of  Coder pairs}\label{seccomplex}
In this section, we will give a proper corepresentation of a Coder pair and construct a complex associated to a Coder pair with coefficients in a corepresentation over it. Let $(M, \rho_l, \rho_r)$ be a bicomodule of a coassociative coalgebra $(C,\Delta)$.

\begin{definition}
Let $(C,\psi_C)$ be a Coder pair. A left comodule of $(C,\psi_C)$ consists of a pair $(M,\psi_M)$ where $(M,\rho_l)$ is a left comodule of $C$ and
$\psi_M \in \mathfrak{gl}(M)$ satisfies
\begin{align}\label{eq-leftcomodpair}
\rho_l\circ\psi_M= ({\Id}_C\otimes\psi_M)\circ\rho_l+(\psi_C\otimes {\Id}_M)\circ\rho_l.
\end{align}
\end{definition}

Similarly, a right comodule of $(C,\psi_C)$ is a pair $(M,\psi_M)$ with respect to  $(M,\rho_r)$ is a right comodule of $C$ and
$\psi_M \in \mathfrak{gl}(M)$ satisfies
\begin{align}\label{eq-rightcomodpair}
\rho_r\circ\psi_M= (\psi_M \otimes {\Id}_C)\circ\rho_r+({\Id}_M\otimes \psi_C)\circ\rho_r.
\end{align}

A bicomodule (or corepresentation) of $(C,\psi_C)$ is a pair $(M,\psi_M)$ which is a left comodule and right comodule of $(C,\psi_C)$ and $M$ is a bicomodule of $C$.
Then $(C,\psi_C)$ is a bicomodule over itself, which is called the {\bf coadjoint} bicomodule.

\begin{remark}\label{rmk Ore}
Given an algebra $A$ with a derivation $d: A\rightarrow A$ and a  left $A$-module $N$ with a map $D: N\rightarrow N$ satisfying
\begin{align}\label{eq-derivation pairmodule}
D(an) = d(a)n + aD(n).
\end{align}
The Ore extension $A[D, d]$ is $A[D]$ as an abelian group, with multiplication induced by $D\cdot a= d(a)+ a\cdot D$ for any $a\in A$.
The Eq {\rm (\ref{eq-rightcomodpair})} in the coalgebraic case is the dual of  Eq {\rm (\ref{eq-derivation pairmodule})} in the algebraic case. Thus, given a Coder pair $(C, \psi_C)$ and a right comodule $(M,\psi_M)$ of $(C, \psi_C)$, an interesting question, which makes us puzzled,  is whether or not one can establish a dual comultiplication of an Ore extension in $C[\psi_M]$.             \hspace{1cm}    $\Box$
\end{remark}

Given a Coder pair $(C,\psi_C)$ and its bicomodule $(M,\psi_M)$, we can construct a semi-direct product Coder pair.
\begin{prop}
Assume $(M,\psi_M)$ is a bicomodule of a Coder pair $(C,\psi_C)$. Then $(C\oplus M,\psi_C+\psi_M)$ is a Coder pair where $C\oplus M$ is a coalgebra where the coassociative structure on $C\oplus M$ is given by the semi-direct coproduct
\begin{eqnarray}\label{eq-semi-coproduct}
\Delta_{C\oplus M}(c,m)=\Delta(c)+\rho_l(m)+\rho_r(m), \ \ c\in C, m\in M.
\end{eqnarray}
\end{prop}
\begin{proof}
It is easy that $(C\oplus M, \Delta_{C\oplus M})$ is a coassociative coalgebra. We only need to show that $\psi_C+\psi_M$ is a coderivation as follows.

\begin{eqnarray*}
&&\Delta_{C\oplus M}(\psi_C+\psi_M)(c, m)\\
&=&\Delta_{C\oplus M}(\psi_C(c), \psi_M(m))\\
&\stackrel{(\ref{eq-semi-coproduct})}{=}&\Delta\psi_C(c)+\rho_l\psi_M(m)+\rho_r\psi(m)\\
&\stackrel{(\ref{eq-coderiCM})(\ref{eq-leftcomodpair})(\ref{eq-rightcomodpair})}{=}&\underline{(\psi_C\otimes 1) \Delta(c)} + (1\otimes \psi_C)\Delta(c)+(1\otimes\psi_M)\rho_l(m)+\underline{(\psi_C\otimes 1)\rho_l(m)}\\
&&+\underline{(\psi_M \otimes 1)\rho_r(m)}+(1\otimes \psi_C)\rho_r(m)\\
&=&\underline{((\psi_C+\psi_M)\otimes 1)(\Delta(c)+\rho_l(m)+\rho_r(m))}\\
&&+(1\otimes(\psi_C+\psi_M))(\Delta(c)+\rho_l(m)+\rho_r(m))\\
&=&((\psi_C+\psi_M)\otimes 1)\Delta_{C\oplus M}(c, m)+(1\otimes(\psi_C+\psi_M))\Delta_{C\oplus M}(c, m),
\end{eqnarray*}
where $c, m\in C\oplus M$.
\end{proof}

Next we recall the classical Hochschild cohomology of the coassociative coalgebra $(C, \Delta)$ with coefficients in $(M,\rho_l, \rho_r)$ as follows(\cite{DY1}\cite{DJ}).

 \indent For $n \geq 0$, the $n$-cochain complex is defined to be
   $$
   C_{c}^{n}(M,C) := \Hom(M, C^{\otimes n})
   $$
with the coboundary operator $\delta_c:C_{c}^{n}(M,C)\rightarrow C_{c}^{n+1}(M,C)$
\begin{eqnarray}\label{eq-diffenrial 1}
\delta_{c}(f)
&=&({\Id}_C \otimes f) \circ \rho_{l}  + \sum_{i=1}^{n} (-1)^{i} \left({\Id}_{C^{\otimes(i-1)}} \otimes \Delta \otimes {\Id}_{C^{\otimes(n-i)}}\right) \circ f  \nonumber\\
&&+ (-1)^{n+1} (f \otimes {\Id}_C) \circ \rho_{r}. \label{eq-diffenrial Hoch}
\end{eqnarray}
for $f \in C^{n}_{c}(M,C)$. Thus $(C_c^{\bullet}(M,C),\delta_c^{\bullet})$ is a cochain complex.  Let $\Zr^n(M,C)$ be the set of $n$-cocycles and $\Br^n(M,C)$ be the set of $n$-coboundaries. The corresponding cohomology groups are denoted by $\Hr^n(M,C)=\Zr^n(M,C)/\Br^n(M,C)$ for $n\geqslant0$, which is called
the Hochschild cohomology groups of $C$ with coefficients in M.

Then we construct the cohomology of a Coder pair by using the above Hochschild cohomology for coassociative coalgebras. Let $(M,\psi_M)$ be  a bicomodule of a Coder pair $(C,\psi_C)$.
 We define the set of Coder pair $0$-cochains to be $C_{\rm Coder}^{0}(M,C)=0$, and define the set of Coder pair $1$-cochains to be
 $C_{\rm Coder}^{1}(M,C)=\Hom(M,C)$, for $n\geqslant2$, we define the set of $n$-cochains by
$$C_{\rm Coder}^{n}(M,C)=C_{c}^{n}(M,C)\times C_{c}^{n-1}(M,C).$$
For $n\geqslant 2$, we define an operator $\omega:C_{c}^{n}(M,C)\rightarrow C_{c}^{n}(M,C)$ by
\begin{align}
\label{omega}
 \omega (f_{n})=\sum_{i=1}^{n}({{\Id}}_{C^{\otimes(i-1)}}\otimes \psi_C \otimes {{{\Id}}}_{C^{\otimes(n-i)}} )\circ f_{n}-f_{n}\circ \psi_M.
\end{align}
\indent Define $d_{c}^{1}:C_{\rm Coder}^{1}(M, C)\rightarrow C_{\rm Coder}^{2}(M, C)$ by
\begin{eqnarray}\label{eq-diff1}
d_{c}^1(f_{1})=(\delta_{c}(f_{1}),(-1)^{1}\omega(f_{1})), \ \ f_1\in \Hom(M,C).
\end{eqnarray}
Then for $n\geqslant2$, we define $d_{c}^{n}:C_{\rm Coder}^{n}(M,C)\rightarrow C_{\rm Coder}^{n+1}(M,C)$ by
$$d_{c}^{n}(f_{n},g_{n-1})=(\delta_{c}(f_{n}),\delta_{c}(g_{n-1})+(-1)^{n}\omega (f_{n})).$$
with $f_{n}\in C_{c}^{n}(M,C),~g_{n-1}\in C_{c}^{n-1}(M,C).$

To prove that $d_{c}\circ d_{c}=0$, the following lemma is needed.
\begin{lemma} \label{lemma-omega}
 The map $\delta_c$ and $\o$ are commutative with each other, that is,
\begin{eqnarray}\label{eq-lemma-omega}
\delta_{c}\circ \omega=\omega\circ\delta_{c}.
\end{eqnarray}
\end{lemma}
\begin{proof}
We firstly have
\begin{align}
&(\delta_{c}\circ \omega)(f_{n})\nonumber \\
%\stackrel{(\ref{eq-diffenrial Hoch})}{=}&({\Id}_{C}\otimes \omega(f_{n}))\circ\rho_l+
%\sum_{i=1}^{n}(-1)^{i}({\Id}_{C^{\otimes(i-1)}}\otimes \Delta\otimes {\Id}_{C^{\otimes(n-i)}} )\circ\omega(f_{n}) \nonumber\\
%&+(-1)^{n+1}(\omega(f_{n})\otimes {\Id}_{C})\circ\rho_r \nonumber\\
%\stackrel{(\ref{omega})}{=}&\big[ {\Id}_{C}\otimes(\sum_{i=1}^{n}({\Id}_{C^{\otimes(i-1)}}\otimes \psi_C\otimes {\Id}_{C^{\otimes(n-i)}})\circ f_{n}-f_{n}\circ \psi_M)\big]\circ\rho_l\nonumber\\
%&+\sum_{i=1}^{n}(-1)^{i}({\Id}_{C^{\otimes(i-1)}}\otimes \Delta\otimes {\Id}_{C^{\otimes(n-i)}} )\circ \nonumber\\
%&(\sum_{j=1}^{n}({\Id}_{C^{\otimes(j-1)}}\otimes \psi_C\otimes {\Id}_{C^{\otimes(n-j)}})\circ f_{n}-f_{n}\circ \psi_M) \nonumber\\
%&+(-1)^{n+1}\big((\sum_{i=1}^{n}({\Id}_{C^{\otimes(i-1)}}\otimes \psi_C\otimes {\Id}_{C^{\otimes(n-i)}})\circ f_{n}-f_{n}\circ \psi_M)\otimes {\Id}_C\big)\circ\rho_r \notag \\
\stackrel{(\ref{eq-diffenrial Hoch})(\ref{omega})}{=}&
\tag{a1}({\Id}_{C}\otimes(\sum_{i=1}^{n}({\Id}_{C^{\otimes(i-1)}}\otimes \psi_C\otimes {\Id}_{C^{\otimes(n-i)}})\circ f_{n})\circ \rho_l- ({\Id}_{C}\otimes f_{n}\c \psi_M)\circ\rho_l \label{a1}\\
&+\tag{a2}\sum_{i=1}^{n}\sum_{j=1}^{n}(-1)^{i}({\Id}_{C^{\otimes(i-1)}}\otimes \Delta\otimes {\Id}_{C^{\otimes(n-i)}} )\circ({\Id}_{C^{\otimes(j-1)}}\otimes \psi_C\otimes {\Id}_{C^{\otimes(n-j)}})\circ f_{n} \label{a2}\\
&\tag{a3}\label{a3}-\sum_{i=1}^{n}(-1)^{i}({{\Id}}_{C^{\otimes(i-1)}}\otimes \Delta\otimes {\Id}_{C^{\otimes(n-i)}} )\circ f_{n}\circ \psi_M \\
&
+(-1)^{n+1}\big( (\sum_{i=1}^{n}({\Id}_{C^{\otimes(i-1)}}\otimes \psi_C\otimes {\Id}_{C^{\otimes(n-i)}})\circ f_{n}\otimes {\Id}_{C} \big)\circ\rho_r \notag\\
&-(-1)^{n+1}(f_{n}\c \psi_M\otimes  {\Id}_{C})\circ\rho_r \label{a4}\tag{a4}
\end{align}
and
\begin{align}
&(\o \c\delta_c) (f_n) =\o (\delta_c (f_n)) \nonumber\\
%\stackrel{(\ref{omega})}{=}&\sum_{i=1}^{n+1}({\Id}_{C^{\t(i-1)}}\t \psi_C \t {\Id}_{C^{\otimes(n+1-i)}})\c\delta_{c}(f_{n})-\delta_{c}(f_{n})\c \psi_M\notag\\
%\stackrel{(\ref{eq-diffenrial Hoch})}{=}&\sum_{i=1}^{n+1}({\Id}_{C^{\t(i-1)}}\t  \psi_C\t {\Id}_{C^{\otimes(n+1-i)}})\c\big[({\Id}_C\t f_n)\c\rho_l \notag \\
%&+\sum_{j=1}^{n}(-1)^j ({\Id}_{C^{\t(j-1)}}\t \d\t {\Id}_{C^{\t(n-j)}})\c f_n+(-1)^{n+1}(f_n\t{\Id}_C)\c \rho_r\big] \notag \\
%&-\Big[({\Id}_C\t f_n)\c\rho_l+\sum_{i=1}^n (-1)^i({\Id}_{C^{\t (i-1)}} \t \d\t{\Id}_{C^{\t (n-i)}})\c f_n \notag\\
%&+(-1)^{n+1}(f_n \t {\Id}_C)\c\rho_r  \Big]\c\psi_M \notag\\
\stackrel{(\ref{eq-diffenrial Hoch})(\ref{omega})}{=}&
\label{b1}\tag{b1}
\sum_{i=1}^{n+1}({\Id}_{C^{\t(i-1)}}\t  \psi_C\t {\Id}_{C^{\otimes(n+1-i)}})\c({\Id}_C\t f_n)\c\rho_l -({\Id}_C\t f_n)\c\rho_l\c \psi_M\\
&+\label{b2}\tag{b2}
\sum_{i=1}^{n+1}\sum_{j=1}^{n}(-1)^j ({\Id}_{C^{\t(i-1)}}\t  \psi_C\t {\Id}_{C^{\t(n+1-i)}})\c ({\Id}_{C^{\t(j-1)}}\t \d\t {\Id}_{C^{\t(n-j)}})\c f_n \\
&-\label{b3}\tag{b3}
\sum_{i=1}^n (-1)^i({\Id}_{C^{\t i-1}} \t \d\t{\Id}_{C^{\t n-i}})\c f_n \c \psi_M \\
&
+\sum_{i=1}^{n+1}(-1)^{n+1}({\Id}_{C^{\t(i-1)}}\t  \psi_C\t {\Id}_{C^{\otimes(n+1-i)}})\c(f_n\t{\Id}_C)\c \rho_r \notag\\
 &-(-1)^{n+1}(f_n \t {\Id}_C)\c\rho_r\c \psi_M.\label{b4}\tag{b4}
\end{align}
Then
\begin{align*}
&(\ref{b1})\\
=&\sum_{i=1}^{n+1}({\Id}_{C^{\t(i-1)}}\t  \psi_C\t {\Id}_{C^{\otimes(n+1-i)}})\c({\Id}_C\t f_n)\c\rho_l -({\Id}_C\t f_n)\c\rho_l\c \psi_M \\
=&(\psi_C\t{\Id}_{C^{\t n}})\c({\Id}_C \t f_n)\c \rho_l\\
&+\big({\Id}_C \t\sum_{i=1}^n({\Id}_{C^{\t (i-1)}}\t \psi_C\t{\Id}_{C^{\t (n-i)}})\big)\c({\Id}_C \t f_n)\c \rho_l-({\Id}_C \t f_n)\c\rho_l\c \psi_M\\
=&(\psi_C\t f_n)\c \rho_l +\big({\Id}_C \t \sum_{i=1}^{n}({\Id}_{C^{\t (i-1)}}\t \psi_C\t{\Id}_{C^{\t (n-i)}})\c f_n\big)\c\rho_l \\
&-({\Id}_C \t f_n)\c\rho_l\c \psi_M\\
=&({\Id}_C \t f_n)(\psi_C\t{\Id}_C)\c\rho_l+\big({\Id}_C \t \sum_{i=1}^{n}({\Id}_{C^{\t (i-1)}}\t \psi_C\t{\Id}_{C^{\t (n-i)}})\c f_n\big)\c\rho_l\\
&-({\Id}_C \t f_n)\c\rho_l\c \psi_M\\
\stackrel{(\ref{eq-leftcomodpair})}{=}&-({\Id}_C \t f_n)({\Id}_C \t \psi_M)\c\rho_l +({\Id}_C\otimes(\sum_{i=1}^{n}({\Id}_{C^{\otimes(i-1)}}\otimes \psi_C\otimes {\Id}_{C^{\otimes(n-i)}})\circ f_{n})\circ \rho_l\\
=&(\ref{a1}).
\end{align*}
Similarly,  we get (\ref{a2})=(\ref{b2}) by (\ref{eq-rightcomodpair}). Obviously, (\ref{a3})=(\ref{b3}). Now we only show that (\ref{b4})=(\ref {a4}). Observe that
\begin{align*}
(\ref{b4})=&\sum_{i=1}^{n+1}(-1)^{n+1}({\Id}_{C^{\t(i-1)}}\t  \psi_C\t {\Id}_{C^{\otimes(n+1-i)}})\c(f_n\t{\Id}_C)\c \rho_r \\
&-(-1)^{n+1}(f_n \t {\Id}_C)\c\rho_r\c \psi_M \\
=&(-1)^{n+1}\big(\sum_{i=1}^{n}({\Id}_{C^{\otimes(i-1)}}\otimes \psi_C\otimes {\Id}_{C^{\otimes(n-i)}})\t {\Id}_C \big)\c(f_n \t {\Id}_C)\c \rho_r\\
&+(-1)^{n+1}({\Id}_{C^{\otimes n}}\otimes \psi_C)\c(f_n \t {\Id}_C)\c \rho_r-(-1)^{n+1}(f_n \t {\Id}_C)\c\rho_r\c \psi_M \\
=&(-1)^{n+1}\big(\sum_{i=1}^{n}({\Id}_{C^{\otimes(i-1)}}\otimes \psi_C\otimes {\Id}_{C^{\otimes(n-i)}})\t {\Id}_C \big)\c(f_n \t {\Id}_C)\c \rho_r\\
&+(-1)^{n+1}(f_n\otimes {\Id}_C)\c({\Id}_M \t \psi_C)\c \rho_r-(-1)^{n+1}(f_n \t {\Id}_C)\c\rho_r\c \psi_M \\
\stackrel{(\ref{eq-rightcomodpair})}{=}&(-1)^{n+1}\big(\sum_{i=1}^{n}({\Id}_{C^{\otimes(i-1)}}\otimes \psi_C\otimes {\Id}_{C^{\otimes(n-i)}})\c f_n\t {\Id}_C \big)\c \rho_r \\
&- (-1)^{n+1}(f_n\otimes {\Id}_C)\c(\psi_M \t {\Id}_C)\c \rho_r\\
=&(\ref{a4}).
\end{align*}
Hence,  $\omega\circ\delta_{c}=\delta_{c}\circ \omega$.
\end{proof}
By Lemma \ref{lemma-omega}, it follows that
\begin{lemma}
The map $d_c$ is a coboundary operator, that is, $d_c\circ d_c=0$.
\end{lemma}

\begin{proof}
For $n\geq1$, we have
\begin{align*}
(d_c\circ d_c)(f_n,g_{n-1})=&d_c(\delta_cf_n,\delta_c g_{n-1}+(-1)^{n}\omega f_n)\\
=&\big(\delta_c(\delta_cf_n),\delta_c(\delta_cg_{n-1}+(-1)^{n}\omega f_n)+(-1)^{n+1}\omega(\delta_cf_n)\big)\\
=&\big(0,(-1)^{n}(\delta_c\omega-\omega\delta_c)f_n\big)\\
\stackrel{(\ref{eq-lemma-omega})}{=}&(0,0).
\end{align*}
Thus, $d_c$ is a coboundary map.
\end{proof}

 Then we get a cochain complex $(C_{\rm Coder}^{\bullet}(M,C),d_c^{\bullet})$ for a bicomodule $(M,\psi_M)$ of a Coder pair $(C,\psi_C)$. Denote the set of $n$-cocycles by $\Zr_{\rm Coder}^{n}(M,C)$ and the set of $n$-coboundaries by $\Br_{\rm Coder}^{n}(M,C)$, then the corresponding cohomology group is defined by $\Hr_{\rm Coder}^{n}(M,C)=\Zr_{\rm Coder}^{n}( M,C)/ \Br_{\rm Coder}^{n}(M,C)$.
\begin{prop}
Assume $(M,\psi_M)$ is a bicomodule of a Coder pair $(C,\psi_C)$. Then
$$\Hr^1_ {\rm Coder}=\{f|f\in\Zr^1(M,C),\psi_c\circ f=f\circ\psi_M \}.$$
\end{prop}
\begin{proof}
For any $f\in C^1_{\rC}(M,C)=\Hom(M,C)$, then $d_{c}^1(f)=(\delta_{c}(f),-\omega(f))$.
Thus, $f\in\Zr^1_{\rC}(M,C)$ if and only if $f\in \Zr^1(M,C)$ and $\psi_c\circ f=f\circ\psi_M$.
\end{proof}
\newpage
\begin{remark}\label{Remark mapping cone}
~\\
\begin{enumerate}
\item[{\rm (i)}]By Lemma {\rm \ref{lemma-omega}}, we have a morphism between two cochain complexes

$$\xymatrix{
C_{c}^{\bullet}(M,C)  :~~~ \cdot\cdot\cdot\ar[r]&{C_{c}^{n}(M,C)}\ar[r]^{\delta^n_c}\ar[d]^{\omega^n}& {C_{c}^{n+1}(M,C)}\ar[r]\ar[d]^{\omega^{n+1}}&\cdot\cdot\cdot\\
C_{c}^{\bullet}(M,C)  :~~~ \cdot\cdot\cdot\ar[r]&{C_{c}^{n}(M,C)}\ar[r]^{\delta^n_c}&{C_{c}^{n+1}(M,C)}\ar[r]&\cdot\cdot\cdot
}.$$
Then $C_{\rm Coder}^{\bullet}(M,C)=C_{c}^{\bullet}(M,C)\oplus C_{c}^{\bullet}(M,C)[-1]$ is a complex with coboundary operator
\begin{align}
d_c^n=\begin{pmatrix}\delta_c^n&0\\(-1)^n\omega^n&\delta_c^{n-1}\end{pmatrix}.
\end{align}
That is, $C_{\rm Coder}^{\bullet}(M,C)$ can be considered as the mapping cone of $\omega^{\bullet}$. There is an obvious short exact sequence of complexes
\begin{align}
0\longrightarrow C_{c}^{\bullet}(M,C)[-1]\longrightarrow C_{\rm Coder}^{\bullet}(M,C)\longrightarrow C_{c}^{\bullet}(M,C)\longrightarrow 0,
\end{align}
which induces a long exact sequence in the cohomology
\begin{align}\label{eq-long exact}
\cdot\cdot\cdot\rightarrow H^{n-1}(C_{c}^{\bullet}(M,C))\rightarrow \Hr_{\rm Coder}^{n}(M,C)\rightarrow H^{n}(C_{c}^{\bullet}(M,C))\rightarrow\cdot\cdot\cdot.
\end{align}
\item[{\rm (ii)}] In the paper, our main object to study is a coassociative coalgebra $C$ with a coderivation $\psi_C$. $\psi_C$ may be considered as a cooperation of a coassociative coalgebra $C$. Can  the  way to construct the complex of a Coder pair $(C, \psi_C)$ be generalized to an general (co)operation? We are not quite sure. But we guess that it is helpful for constructing a complex of an operator algebra to build a suitable chain map being analogous to  $\omega^\bullet$. For example, recently a cohomology complex of a Rota-Baxter algebra has been established via a suitable chain map {\rm (See \cite[Proposition 6.2]{Wang})}.
\end{enumerate}
\end{remark}

\begin{example}
~~
\begin{enumerate}
\item[{\rm (i)}] In \cite{DY1}, $C$ is coseparable if and only if $C$ is injective as a $C\otimes C^{op}$-comodule. If $C$ is coseparable, then $\Hr^n(M, C)={0}, n\geqslant 1$ for any $C$-bicomodule $M$ {\rm (See \cite[Theorem 3]{DY1})}. Given a coseparable coalgebra $C$ with a coderivation $\psi_C$ and a bicomodule $(M, \psi_M)$ of Coder pair $(C, \psi_C)$, we have that $\Hr_{\rm Coder}^{2}(M, C) = 0$ via Eq. (\ref{eq-long exact}).

\item[{\rm (ii)}]Let $(T^C(V)=T^C, \Delta)$ be a tensor coalgebra over $\field$, whose underlying space is the same as the tensor algebra and the  coproduct is given by
\begin{align*}
\Delta(v_1\otimes v_2\otimes \cdot\cdot\cdot\otimes v_n)=\sum_{i+j=n}\varphi_{i,j}^{-1}(v_1\otimes v_2\otimes \cdot\cdot\cdot\otimes v_n)\in \bigoplus_{i=0}^nT^iV\otimes T^{n-i}V,
\end{align*}
where $\varphi_{i,j}: T^iV\otimes T^j V\rightarrow T^{i+j}V$ is the canonical isomorphism. The tensor coalgebra has only non trivial cohomology $\Hr^0(C_{c}^{\bullet}(T^C,T^C))$ and $\Hr^1(C_{c}^{\bullet}(T^C,T^C))$. Due to Eq. (\ref{eq-long exact}), we have that $\Hr_{\rm Coder}^{3}(T^C,T^C) = 0$ for a Coder pair $(T^C, \psi_{T^C})$.
\end{enumerate}
\end{example}

\section{Deformations of  Coder pairs}\label{secdefomation}
In this section, formal deformations of  a Coder pair and extensions of finite order deformations are studied.
\subsection{Formal deformations}
Let $(C,\d)$ be a coalgebra and $\psi_C$ be a coderivation of $C$, then $(C,\psi_C)$ is a Coder pair.
A deformation of a coalgebra $C$ is a power serie $\Delta_t = \sum_{n=0}^\infty \Delta_n t^n$ in which each $\Delta_n \in C^2_c(C,C)$ with $\Delta_0 = \Delta$, such that $\Delta_t$ is coassociative: $({\Id}_C \otimes \Delta_t) \c\Delta_t = (\Delta_t \otimes {\Id}_C)\c \Delta_t$.
In particular, $\Delta_t$ gives an $\mathbb{F} [[ t ]]$-coalgebra structure on the module of power series $C [[ t ]]$ with coefficients in $C$ that restricts to the original coalgebra structure on $C$ when setting $t = 0$.

Consider a $t$-parametrized family of linear maps
\begin{align*}
\d_{t}&=\sum_{i\geq0}\d_{i} t^{i},\quad \d_{i}\in C_{c}^2(C, C),~\d_{0}=\d,\\
\psi_{t}&=\sum_{i\geq0}\psi_{i} t^{i},\quad \psi_{i}\in C_{c}^{1}(C,C),~\psi_{0}=\psi_C.
\end{align*}
\begin{definition}
If $\Omega_t=\big(\d_{t},\psi_{t}\big)$ satisfies
\begin{align}
\label{def1}&({\Id}_C \otimes \Delta_t) \c\Delta_t = (\Delta_t \otimes {\Id}_C) \c\Delta_t,\\
&\Delta_t \c \psi_t=(\psi_t\otimes {\Id}_C)\circ \Delta_t + ({\Id}_C\otimes \psi_t)\circ \Delta_t .\label{def2}
\end{align}
Then we say that  $\Omega_t=\big(\d_{t},\psi_{t}\big)$ is a 1-parameter formal deformation of the Coder pair $(C,\psi_C)$.
\end{definition}
If $\Omega_t$ is a 1-parameter formal  deformation of a Coder pair $(C,\psi_C)$, then $(C[[t]], \d_t)$ is a coalgebra, and $\psi_{t}$ is a coderivation of $C[[t]]$.
For convenience, we will write $\Omega_t=\sum_{i=0}^{\infty}\Omega_i t^i$ with $\Omega_i=(\d_i,\psi_i)\in C_{\rC}^{2}(C,C).$\\
\indent Expanding the equations (\ref{def1}), (\ref{def2}) and collecting coefficients of $t^{k}$, we see that (\ref{def1}), (\ref{def2}) are equivalent to the system of equations
\begin{align}
\label{def3}&\sum\limits_{\begin{subarray}{l}i+j=k\\i,j\geq0\end{subarray}}({\Id}_C \otimes \Delta_i) \c\Delta_j - (\Delta_i \otimes {\Id}_C) \c\Delta_j=0,\\
\label{def4}&\sum\limits_{\begin{subarray}{l}i+j=k\\i,j\geq0\end{subarray}}\Delta_i \c \psi_j-(\psi_i\otimes {\Id}_C)\circ \Delta_j - ({\Id}_C\otimes \psi_i)\circ \Delta_j=0.
\end{align}

\begin{remark}
For $n=0$, (\ref{def3}) and (\ref{def4}) both holds automatically since $(C,\psi_C)$ is a Coder pair.
\end{remark}

\begin{prop}\label{z2}
Let $\big(\d_t,\psi_t)$ be a 1-parameter formal deformation of the Coder pair $(C,\psi_C)$, then
$\big(\d_1,\psi_1)\in C^2_{\rC}(C,C)$ is a $2$-cocycle  in the cohomology of the Coder pair $(C,\psi_C)$ with coefficients in itself.
Moreover, if $(\d_i,\psi_i)=0$, for $ i\geqslant 1$, then $(\d_{i+1},\psi_{i+1})$ is a $2$-cocycle in $C_{\rC}^{2}(C,C)$.
\end{prop}
\begin{proof}
Let $k=1$, then (\ref{def3}) is equal to
\begin{align*}
({\Id}_C \t \d_1)\c \d + ({\Id}_C \t \d)\c \d_1 -(\d_1\t {\Id}_C)\c \d -(\d\t {\Id}_C)\c \d_1=0.
\end{align*}
That is to say $\delta_c(\d_{1})=0$, and (\ref{def4}) is equal to
\begin{align*}
&\Delta_1 \c \psi_C-(\psi_1\otimes {\Id}_C)\circ \Delta - ({\Id}_C\otimes \psi_1)\circ \Delta+\Delta \c \psi_1\\
&-(\psi_C\otimes {\Id}_C)\circ \Delta_1 - ({\Id}_C\otimes \psi_C)\circ \Delta_1=0.
\end{align*}
It follows that $\delta_c(\psi_{1})+\omega(\d_{1})=0$,
thus we have
$$d_{c}^{2}\big(\d_{1},\psi_{1}\big)=\big(\delta_c(\d_{1}),\delta_c(\psi_{1})+\omega(\d_{1}))=0.$$
That is, $\big(\d_1,\psi_1)$ is a $2$-cocycle in $C_{\rC}^{2}(C,C)$.
\end{proof}
\begin{definition}
The $2$-cocycle $(\Delta_1,\psi_1)$ is called the infinitesimal of the 1-parameter formal deformation $(\Delta_t, \psi_t)$ of the Coder pair $(C,\psi_C)$.
\end{definition}
\begin{definition}
Let $\big(\d_{t},\psi_{t}\big)$ and $\big(\overline{\d}_{t},\overline{\psi}_{t}\big)$ be two deformations of $(C,\psi_C)$. A formal automorphism from $\big(\d_{t},\psi_{t}\big)$ to $\big(\overline{\d}_{t},\overline{\psi}_{t}\big)$ is a power series $\Phi_{t}=\sum_{i\geq0}\phi_{i}t^{i}:C[[t]]\rightarrow C[[t]]$, where $\phi_{i}\in C_c^{1}(C,C)$ with $\phi_{0}={\Id}_C$, such that
\begin{align}
 \overline{\d}_t\c\Phi_{t}&=(\Phi_t \t \Phi_t)\c \d_t\label{ed1},\\
 \overline{\psi}_{t}\c \Phi_{t}&=\Phi_{t}\c \psi_{t}\label{ed2}.
\end{align}

Two 1-parameter deformations  $\big(\d_{t},\psi_{t}\big)$  and $\big(\overline{\d}_{t},\overline{\psi}_{t}\big)$ are said to be equivalent if there exists a formal automorphism $\Phi_{t}:\big(\d_{t},\psi_{t}\big)\rightarrow\big(\overline{\d}_{t},\overline{\psi}_{t}\big)$.
\end{definition}
\begin{theorem}
The infinitesimals of two equivalent 1-parameter deformations of Coder pair $(C,\psi_C)$ are in the same cohomology class.
\end{theorem}
\begin{proof}
Let $\Phi_{t}:\big(\d_{t},\psi_{t}\big)\rightarrow\big(\overline{\d}_{t},\overline{\psi}_{t}\big)$ be a formal automorphism, then compare coefficients of $t$ in (\ref{ed1}) and (\ref{ed2}), we have
\begin{align*}
\overline{\d}_{1}+\d \c\phi_{1}&=\d_{1}+({\Id}_C \t \phi_1)\c \d+(\phi_1 \t {\Id}_C)\c \d,
\\
\overline{\psi}_{1}+ \psi_C\c \phi_{1}&=\psi_{1}+\phi_{1}\c \psi_C.
\end{align*}
Thus we have
$$\big(\overline{\d}_{1},\overline{\psi}_{1}\big)=\big(\d_{1},\psi_{1}\big)+d_{c}^{1}(\phi_{1}),$$
which implies that $\big(\overline{\d}_{1},\overline{\psi}_{1}\big)$ and $\big(\d_{1},\psi_{1}\big)$ are in the same cohomology class.
\end{proof}
\begin{definition}
A 1-parameter fomal deformation $\big(\d_t,\psi_t\big)$ of a Coder pair $(C,\psi_C)$ is said to be trivial if it is equivalent to $\big(\d,\psi_C\big)$.
\end{definition}
\begin{definition}
A Coder pair $(C,\psi_C)$ is said to be rigid if every 1-parameter formal deformation of $(C,\psi_C)$ is trivial.
\end{definition}
\begin{theorem}\label{Thm-rigid}
If $\Hr_{\rC}^{2}(C,C)=0$, then the Coder pair $(C,\psi_C)$ is rigid.
\end{theorem}
\begin{proof}
Let $(\d_t,\psi_t)$ be a 1-parameter formal deformation of the Coder pair $(C,\psi_C)$. By Proposition \ref{z2}, $(\d_1,\psi_1) \in \Zr_{\rC}^{2}(C,C)$. By $\Hr_{\rC}^{2}(C,C)=0$, there exists a 1-cochain $\phi_1\in  C_{\rC}^{1}(C,C)$ such that
$$(\d_1,\psi_1)=-d_c(\phi_1)=-(\delta_c \phi_1, -\omega \phi_1).$$
Then setting $\Phi_t={{\Id}_A}+\phi_1 t$, we have a deformation $(\overline{\d}_t,\overline{\psi}_t)$ such that
\begin{align*}
\overline{\d}_t&=(\Phi_t \t \Phi_t)\c \d_t\c \Phi_{t}^{-1},\\
 \overline{\psi}_{t}&= \Phi_{t}\c \psi_{t}\circ\Phi_{t}^{-1}.
\end{align*}
Thus, $(\d_t,\psi_t)$ is equivalent to $(\overline{\d}_t,\overline{\psi}_t)$. Moreover, we have
\begin{eqnarray*}
\overline{\d}_t&=&\big(({\Id}_C +\phi_1 t)\t({\Id}_C +\phi_1 t)\big)\c\d_t \c ({\Id}_C-\phi_1t+\phi_1^2t^{2}+\cdots+(-1)^i\phi_1^it^{i}+\cdots),\\
\overline{\psi}_t&=&({\Id}_C +\phi_1 t)\c \psi_t \c ({\Id}_C-\phi_1t+\phi_1^2t^{2}+\cdots+(-1)^i\phi_1^it^{i}+\cdots).
\end{eqnarray*}
Thus we have
\begin{eqnarray*}
\overline{\d}_t&=&\d+(\d_1+({\Id}_C\t \phi_1)\c\d+(\phi_1\t{\Id}_C)\c\d-\d\c\phi_1)t+\overline{\d}_{2}t^{2}+\cdots,\\
\overline{\psi}_t&=&\psi_C+(\psi_1-\psi_C\c\phi_1+\phi_1\c \psi_C)t+\overline{\psi}_{2}t^{2}+\cdots.
\end{eqnarray*}
By $(\d_1,\psi_1)=-d_c(\phi_1)$, we get
\begin{eqnarray*}
\overline{\d}_t&=&\d+\overline{\d}_{2}t^{2}+\cdots,\\
\overline{\psi}_t&=&\psi_C+\overline{\psi}_{2}t^{2}+\cdots.
\end{eqnarray*}
Then, by repeating the argument, we can show that $(\d_t,\psi_t)$ is equivalent to $(\d,\psi_C)$.
\end{proof}

\subsection{Obstructions}
Given a $2$-cocycle $\Omega_1=(\d_1,\psi_1)\in C_{\rC}^{\bullet}(C,C)$, is there a 1-parameter formal deformation of a Coder pair $(C,\psi_C)$ with $\Omega_1$ as its infinitesimal?
The purpose of this subsection is to identify the obstructions for the existence of such a deformation. If such a deformation exists, then $\Omega_1$ is said to be integrable.
\begin{definition}\label{nd}
Let $n\geq1$, a deformation of order $n$ of a Coder pair $(C,\psi_C)$ is a pair $\Omega_t=\big(\d_{t},\psi_{t}\big)$ with $\d_{t}=\sum_{i=0}^{n}\d_{i}t^{i}$ , $\psi_{t}=\sum_{i=0}^{n}\psi_{i} t^{i}$, such that $(\d_{0},\psi_{0}\big)=(\d,\psi_C\big)$ and satisfying the following equations
\begin{align}
\label{def1.1}&({\Id}_C \otimes \Delta_t) \c\Delta_t = (\Delta_t \otimes {\Id}_C) \c\Delta_t  &(\mathrm{mod}~t^{n+1}) ,\\
&\Delta_t \c \psi_t=(\psi_t\otimes {\Id}_C)\circ \Delta_t + ({\Id}_C\otimes \psi_t)\circ \Delta_t &(\mathrm{mod}~ t^{n+1}).\label{def2.1}
\end{align}
\end{definition}
In other words, $\d_t$ defines an $\mathbb{F}[[t]]/(t^{n+1})$-coalgebra structure on $C[[t]]/(t^{n+1})$, and $\psi_t$ is a coderivation of a coalgebra $C[[t]]/(t^{n+1})$.

To answer the integrability question, it suffices to consider the obstruction to extending a deformation $\Omega_t$  of order $n$ to a deformation of $(C,\psi_C)$ of order $n+1$. Then let $\Omega_{n+1}=(\d_{n+1},\psi_{n+1})\in C_{\rC}^{2}(C,C)$ be a $2$-cochain, set $\Omega'_t=\Omega_t+\Omega_{n+1}t^{n+1}$, is $\Omega'_t$ a deformation of $(C,\psi_C)$ of order $n+1$? Since $\Omega'_t=\Omega_t ~~ (\mathrm{mod}~t^{n+1})$, it sufficies to consider the coefficients of $t^{n+1}$ in the deformation of $(C,\psi_C)$.
\begin{definition}
Let $\Omega_t=\big(\d_{t},\psi_{t}\big)$ be a deformation of order $n$ of a Coder pair $(C,\psi_C)$. If there exists a $2$-cochain $\big(\d_{n+1},\psi_{n+1}\big) \in C_{\rm Coder}^{2}(C,C)$, such that the pair $\overline{\Omega}_t=\big(\overline{\d}_{t},\overline{\psi}_{t}\big)$ with
$$\overline{\d}_{t}=\d_{t}+\d_{n+1}t^{n+1},~ \overline{\psi}_{t}=\psi_{t}+\psi_{n+1}t^{n+1}. $$
is a deformation of order $n+1$ of $(C,\psi_C)$, then we say that $\Omega_t$ is extensible.
\end{definition}
Let $\Omega_t=\big(\d_{t},\psi_{t}\big)$ be a deformation of order $n$ of a Coder pair $(C,\psi_C)$, we define  ${\Ob}_\Omega=\big({\Ob}_{C}^{3},{\Ob}_{\psi}^{2}\big)\in C_{\rm Coder}^{3}(C,C)$ by
\begin{align}
\label{eq-obC}{\Ob}_{C}^{3}=&\sum\limits_{i+j=n+1\atop i,j>0}
\Big((\d_i\t{\Id}_C) \c \d_j-({\Id}_C \t \d_i)\c \d_j\Big),\\
\label{eq-obpsi}{\Ob}_{\psi}^{2}=&\sum\limits_{i+j=n+1\atop i,j>0}
\Big(\d_i \c \psi_j -({\Id}_C \t \psi_i)\c \d_j -(\psi_i \t {\Id}_C)\c \d_j
\Big).
\end{align}
\begin{prop}\label{ext2}
Let $\Omega_t=\big(\d_{t},\psi_{t}\big)$ be a deformation of order $n$ of a Coder pair $(C,\psi_C)$, then the 3-cochain ${\Ob}_\Omega=\big({\Ob}_{C}^{3},{\Ob}_{\psi}^{2}\big)$ is a $3$-cocycle of the Coder pair $(C,\psi_C)$ with coefficient in itself.
\end{prop}
Before giving the proof of Proposition \ref{ext2}, we give the following definition and two immediate consequences.

\begin{definition}
Let $\Omega_t=\big(\d_{t},\psi_{t}\big)$ be a deformation of order $n$ of a Coder pair $(C,\psi_C)$. The cohomology class $[({\Ob}_\Omega=\big({\Ob}_{C}^{3},{\Ob}_{\psi}^{2})]\in \Hr_{\rm Coder}^3(C,C)$ is called the obstruction class of $\Omega_t$ being extensible.
\end{definition}
\begin{theorem}\label{ext1}
Let $\Omega_t=\big(\d_{t},\psi_{t}\big)$ be a deformation of order $n$ of a Coder pair $(C,\psi_C)$, then $\Omega_t$ is extensible if and only if the obstruction class
$[\big({\Ob}_{C}^{3},{\Ob}_{\psi}^{2}\big)]$ is trivial.
\end{theorem}
\begin{proof}
Suppose a deformation $\Omega_t=\big(\d_{t},\psi_{t}\big)$ of order $n$ of $(C,\psi_C)$  extends to a deformation $\overline{\Omega}_t=\Omega_t +(\d_{n+1},\psi_{n+1})t^{n+1}$ of order $n+1$ of $(C,\psi_C)$. Then (\ref{def1.1}) and (\ref{def2.1}) hold for $\overline{\Omega}_t$, compare the coefficients of $t^{n+1}$ in (\ref{def1.1}) and (\ref{def2.1}), we get
$${\Ob}_{C}^{3}=\delta_c(\d_{n+1}),\quad {\Ob}_{\psi}^{2} =\delta_c (\psi_{n+1})+\omega(\d_{n+1}).$$
which implies that $$\big({\Ob}_{C}^{3},{\Ob}_{\psi}^{2}\big)=d_c(\d_{n+1},\psi_{n+1}).$$
Thus $\big({\Ob}_{C}^{3},{\Ob}_{\psi}^{2}\big)\in B_{\rC}^{3}(C,C)$, which follows that $[\big({\Ob}_{C}^{3},{\Ob}_{\psi}^{2}\big)]$ is trivial.\\
\indent Conversely, let $\big({\Ob}_{C}^{3},{\Ob}_{\psi}^{2}\big)\in B_{\rC}^{3}(C,C)$, and suppose $$\big({\Ob}_{C}^{3},{\Ob}_{\psi}^{2}\big)=d_c(\d_{n+1},\psi_{n+1})$$
for some $2$-cochain $(\d_{n+1},\psi_{n+1})\in C_{\rC}^{2}(C,C)$. Set
$$\overline{\Omega}_t=\Omega_t +(\d_{n+1},\psi_{n+1}).$$
Therefore, $\overline{\Omega}_t$ is a deformation of order $n+1$, i.e., $\Omega_t$ is extensible.
\end{proof}
\begin{cor}\label{ext3}
If $\Hr_{\rC}^3(C,C)=0$, then every $2$-cocycle is the infinitesimal of a 1-parameter formal deformation of a Coder pair $(C,\psi_C)$ .
\end{cor}

\textbf{Proof of Proposition \ref{ext2}}. We must show that
$d_c\big({\Ob}_{C}^{3},{\Ob}_{\psi}^{2}\big)=0$, i.e., $\big(\delta_c ({\Ob}_{C}^{3}),\delta_c  ({\Ob}_{\psi}^{2})-\omega ({\Ob}_{C}^{3})\big)=0$.  One has already known that $\delta_c ({\Ob}_{C}^{3})=0$ (See \cite[Theorem 3]{GS4}). So it remains to show that
$$\delta_c  ({\Ob}_{\psi}^{2})-\omega ({\Ob}_{C}^{3})=0.$$
To do this, we need to show
\begin{align}
&-\omega ({\Ob}_{C}^{3})\notag\\
\stackrel{(\ref{omega})}{=}&{\Ob}_{C}^{3}\c \psi_C-(\psi_C\t {\Id}_C \t {\Id}_C)\c {\Ob}_{C}^{3} -( {\Id}_C \t \psi_C\t {\Id}_C)\c {\Ob}_{C}^{3} \notag \\
&-( {\Id}_C \t{\Id}_C \t \psi_C)\c {\Ob}_{C}^{3} \notag\\
\stackrel{(\ref{eq-obC})}=&\label{e1}\tag{e1}
\sum_{i+j=n+1\atop i,j>0}(\d_i \t {\Id}_C)\c \d_j \c \psi_C\\
&-\label{e2}\tag{e2}
\sum_{i+j=n+1\atop i,j>0}({\Id}_C\t \d_i )\c \d_j \c \psi_C\\
&\begin{array}{l}\label{e3}\tag{e3}
-(\psi_C\t {\Id}_C \t {\Id}_C)\c\sum\limits_{i+j=n+1\atop i,j>0}\big[(\d_i \t {\Id}_C)\c \d_j-({\Id}_C \t \d_i)\c \d_j \big] \\
-( {\Id}_C \t \psi_C\t {\Id}_C)\c \sum\limits_{i+j=n+1\atop i,j>0}\big[(\d_i \t {\Id}_C)\c \d_j-( {\Id}_C\t \d_i)\c \d_j \big] \\
-( {\Id}_C \t{\Id}_C \t \psi_C)\c \sum\limits_{i+j=n+1\atop i,j>0}\big[(\d_i \t {\Id}_C)\c \d_j-( {\Id}_C\t \d_i)\c \d_j \big]
\end{array}
\end{align}
and
\begin{align*}
\delta_c({\Ob}_{\psi}^{2})\stackrel{(\ref{eq-diffenrial Hoch})}=&({\Id}_C \t{\Ob}_{\psi}^{2} )\c \d-(\d \t {\Id}_C )\c {\Ob}_{\psi}^{2}\\
&+({\Id}_C  \t \d)\c {\Ob}_{\psi}^{2} -({\Ob}_{\psi}^{2} \t {\Id}_C)\c\d.
\end{align*}
\indent Then we have
\begin{align}
&({\Id}_C \t{\Ob}_{\psi}^{2} )\c \d \notag\\
\stackrel{(\ref{eq-obpsi})}=&\Big( {\Id}_C \t \sum_{i+j=n+1 \atop i,j>0}\big(  \d_i\c \psi_j-({\Id}_C\t \psi_j)\c\d_i-(\psi_j\t{\Id}_C) \c\d_i \big)   \Big)\c \d \notag\\
=&\label{f1}\tag{f1}
\sum_{i+j=n+1 \atop i,j>0}({\Id}_C \t \d_i\c \psi_j)\c \d \\
&-\label{f2}\tag{f2}
\sum_{i+j=n+1 \atop i,j>0}({\Id}_C \t {\Id}_C\t \psi_j)\c({\Id}_C \t \d_i)\c \d \\
&-\label{f3}\tag{f3}
\sum_{i+j=n+1 \atop i,j>0}({\Id}_C \t \psi_j \t{\Id}_C)\c({\Id}_C \t \d_i)\c \d,
\end{align}
\begin{align}
&-(\d \t {\Id}_C )\c {\Ob}_{\psi}^{2} \nonumber\\
\stackrel{(\ref{eq-obpsi})}=&-(\d \t {\Id}_C)\c \sum_{i+j=n+1 \atop i,j>0}\big(\d_i \c \psi_j-({\Id}_C \t \psi_j)\c \d_i - (\psi_j \t{\Id}_C)\c \d_i    \big) \notag\\
%=&-\sum_{i+j=n+1 \atop i,j>0}(\d \t {\Id}_C )\c\d_i \c \psi_j \notag\\
%&+\sum_{i+j=n+1 \atop i,j>0}(\d \t {\Id}_C)\c({\Id}_C \t \psi_j)\c \d_i \notag\\
%&+\sum_{i+j=n+1 \atop i,j>0}(\d \t {\Id}_C )\c(\psi_j\t{\Id}_C )\c \d_i \notag\\
=&\label{f4}\tag{f4}
-\sum_{i+j=n+1 \atop i,j>0}(\d \t {\Id}_C )\c\d_i \c \psi_j \\
&+\label{f5}\tag{f5}
\sum_{i+q=n+1 \atop i,q>0}({\Id}_C\t {\Id}_C \t \psi_q )\c(\d \t {\Id}_C)\c \d_i \\
&+\label{f6}\tag{f6}
\sum_{i+j=n+1 \atop i,j>0}(\d \t {\Id}_C )\c(\psi_j\t{\Id}_C )\c \d_i,
\end{align}
\begin{align}
&({\Id}_C  \t \d)\c {\Ob}_{\psi}^{2} \notag\\
\stackrel{(\ref{eq-obpsi})}{=}&
\sum_{i+j=n+1 \atop i,j>0}( {\Id}_C\t \d )\c\d_i \c \psi_j \notag\\
&-
\sum_{i+j=n+1 \atop i,j>0}( {\Id}_C\t \d  )\c({\Id}_C \t \psi_j)\c \d_i\notag\\
&-
\sum_{i+j=n+1 \atop i,j>0}( {\Id}_C \t\d  )\c(\psi_j \t {\Id}_C )\c \d_i\notag\\
=&\label{f7}\tag{f7}
\sum_{i+j=n+1 \atop i,j>0}( {\Id}_C\t \d )\c\d_i \c \psi_j \\
&-\label{f8}\tag{f8}
\sum_{i+j=n+1 \atop i,j>0}( {\Id}_C\t \d  )\c({\Id}_C \t \psi_j)\c \d_i\\
&-\label{f9}\tag{f9}
\sum_{i+j=n+1 \atop i,j>0}(\psi_j \t {\Id}_C \t {\Id}_C)\c ( {\Id}_C \t \d)\c\d_i
\end{align}
as well as
\begin{align}
&-({\Ob}_{\psi}^{2} \t {\Id}_C)\c\d \notag\\
=&-\label{f10}\tag{f10}
\sum_{i+j=n+1 \atop i,j>0}(\d_i \c \psi_j \t {\Id}_C )\c \d\\
&+\label{f11}\tag{f11}
\sum_{i+j=n+1 \atop i,j>0}({\Id}_C \t \psi_j \t {\Id}_C)\c(\d_i \t {\Id}_C )\c \d \\
&+\label{f12}\tag{f12}
\sum_{i+j=n+1 \atop i,j>0}(\psi_j \t {\Id}_C \t {\Id}_C)\c(\d_i \t {\Id}_C )\c \d.
\end{align}
\indent By (\ref{def2.1}), for $m\leq n$, we get
$$\sum_{k+l=m \atop k,l\geq0}\big(\Delta_k \c \psi_l-(\psi_l\otimes {\Id}_C)\circ \Delta_k - ({\Id}_C\otimes \psi_l ) \circ \Delta_k \big)=0,$$
which implies that
\begin{align}\label{eq-f6'}
\Delta \c \psi_m=\sum_{k+l=m \atop k,l\geq0}\big((\psi_l\otimes {\Id}_C)\circ \Delta_k + ({\Id}_C\otimes \psi_l)\circ \Delta_k \big)-\sum_{k+l=m \atop k>0,l\geq0}\Delta_k \c \psi_l.
\end{align}
So
\begin{align}
(\ref{f6})=&\sum_{i+j=n+1 \atop i,j>0}(\d \t {\Id}_C )\c(\psi_j\t{\Id}_C )\c \d_i \notag \\
=&\sum_{i+j=n+1 \atop i,j>0}(\d\c \psi_j \t {\Id}_C )\c \d_i \notag \\
\stackrel{(\ref{eq-f6'})}{=}&\sum\limits_{\footnotesize\begin{subarray}{c}p+q+i=n+1\\ p,q\geq0,i>0\\p+q>0\end{subarray}}\Big(\big((\psi_q\otimes {\Id}_C)\circ \Delta_p + ({\Id}_C\otimes \psi_q)\circ \Delta_p \big)\t {\Id}_C\Big)\c \d_i \notag\\
&-\sum\limits_{\footnotesize\begin{subarray}{c}p+q+i=n+1\\ p,i>0\\q\geq0\end{subarray}}(\d_p \c \psi_q\t {\Id}_C)\c \d_i \notag \\
=&\label{f13}\tag{f13}
\sum\limits_{\footnotesize\begin{subarray}{c}p+i=n+1\\ p,i>0\end{subarray}}\Big[\big((\psi_C\otimes {\Id}_C)\circ \Delta_p + ({\Id}_C\otimes \psi_C)\circ \Delta_p \big)\t {\Id}_C\Big]\c \d_i \\
&+\label{f14}\tag{f14}
\sum\limits_{\footnotesize\begin{subarray}{c}p+q+i=n+1\\ p\geq0,~q,i>0\end{subarray}}\Big[\big((\psi_q\otimes {\Id}_C)\circ \Delta_p + ({\Id}_C\otimes \psi_q)\circ \Delta_p \big)\t {\Id}_C\Big]\c \d_i\\
&-\label{f15}\tag{f15}
\sum\limits_{\footnotesize\begin{subarray}{c}p+q+i=n+1\\ p,i>0,~q\geq0\end{subarray}}(\d_p \c \psi_q\t {\Id}_A)\c \d_i.
\end{align}
\begin{align}
(\ref{f8})=&-\sum_{i+j=n+1 \atop i,j>0}( {\Id}_C \t\d  )\c({\Id}_C \t \psi_j)\c \d_i \notag \\
=&-\sum_{i+j=n+1 \atop i,j>0}({\Id}_C \t \d \c \psi_j)\c \d_i \notag \\
\stackrel{(\ref{eq-f6'})}{=}&-\sum\limits_{\footnotesize\begin{subarray}{c}p+q+i=n+1\\ p,q\geq0,i>0\\p+q>0\end{subarray}}\Big[{\Id}_C \t\big((\psi_q\otimes {\Id}_C)\circ \Delta_p + ({\Id}_C\otimes \psi_q)\circ \Delta_p \big)\Big]\c \d_i \notag\\
&+\sum\limits_{\footnotesize\begin{subarray}{c}p+q+i=n+1\\ p,i>0\\q\geq0\end{subarray}}({\Id}_C\t\d_p \c \psi_q)\c \d_i \notag \\
=&-\label{f16}\tag{f16}
\sum\limits_{\footnotesize\begin{subarray}{c}p+i=n+1\\ p,i>0\end{subarray}}\Big[{\Id}_C \t\big((\psi_C\otimes {\Id}_C)\circ \Delta_p + ({\Id}_C\otimes \psi_C)\circ \Delta_p \big)\Big]\c \d_i \\
&-\label{f17}\tag{f17}
\sum\limits_{\footnotesize\begin{subarray}{c}p+q+i=n+1\\ p\geq0~i,q>0\end{subarray}}\Big[{\Id}_C \t\big((\psi_q\otimes {\Id}_C)\circ \Delta_p + ({\Id}_C\otimes \psi_q)\circ \Delta_p \big)\Big]\c \d_i\\
&+\label{f18}\tag{f18}
\sum\limits_{\footnotesize\begin{subarray}{c}p+q+i=n+1\\ p,i>0\\q\geq0\end{subarray}}({\Id}_C\t\d_p \c \psi_q)\c \d_i.
\end{align}
\indent By  (\ref{def3}) and (\ref{def1.1}), for $m\leq n$, we get
\begin{align*}
\sum_{k+l=m \atop k,l\geq 0}({\Id}_C \otimes \Delta_k) \c\Delta_l = \sum_{k+l=m \atop k,l\geq 0} (\Delta_k \otimes {\Id}_C) \c\Delta_l
\end{align*}
That is,
\begin{align}\label{eq-f7'}
({\Id}_C \otimes \Delta) \c\Delta_m=\sum_{k+l=m \atop k,l\geq 0}(\Delta_k \otimes {\Id}_C) \c\Delta_l-\sum_{k+l=m \atop k>0,l\geq 0}({\Id}_C \otimes \Delta_k) \c\Delta_l.
\end{align}
So
\begin{align}
(\ref{f7})=&\sum_{i+j=n+1 \atop i,j>0}( {\Id}_C \t\d )\c\d_i \c \psi_j \notag\\
\stackrel{(\ref{eq-f7'})}{=}&\sum\limits_{\footnotesize\begin{subarray}{c}p+q+j=n+1\\ p,q\geq0,j>0\\p+q>0\end{subarray}}(\d_p \t {\Id}_C )\c \d_q \c \psi_j-\sum\limits_{\footnotesize\begin{subarray}{c}p+q+j=n+1\\ p,j>0,q\geq0\end{subarray}}( {\Id}_C \t \d_p )\c \d_q \c \psi_j \notag\\
=&\label{f19}\tag{f19}
\sum\limits_{\footnotesize\begin{subarray}{c}q+j=n+1\\ q,j>0\end{subarray}}(\d \t {\Id}_C )\c \d_q \c \psi_j \\
&+\label{f20}\tag{f20}
\sum\limits_{\footnotesize\begin{subarray}{c}p+q+j=n+1\\ p,j>0,q\geq0\end{subarray}}(\d_p \t {\Id}_C )\c \d_q \c \psi_j \\
&-\label{f21}\tag{f21}
\sum\limits_{\footnotesize\begin{subarray}{c}p+q+j=n+1\\ p,j>0,q\geq0\end{subarray}}( {\Id}_C \t \d_p )\c \d_q \c \psi_j.
\end{align}
Then we have
\begin{align}
&(\ref{f20})\notag \\
=&\tag{f22}\label{f22}
\sum\limits_{\footnotesize\begin{subarray}{c}p+q+j=n+1\\ p>0,~j,q\geq0\end{subarray}}(\d_p \t {\Id}_C )\c \d_q \c \psi_j \\
&-\tag{f23}\label{f23}
\sum\limits_{\footnotesize\begin{subarray}{c}p+q=n+1\\ p,q>0\end{subarray}}(\d_p \t {\Id}_C )\c \d_q \c \psi_C.
\end{align}
\begin{align}
&(\ref{f21})\notag \\
=&-\tag{f24}\label{f24}
\sum\limits_{\footnotesize\begin{subarray}{c}p+q+j=n+1\\ p>0,~j,q\geq0\end{subarray}}( {\Id}_C \t \d_p )\c \d_q \c \psi_j \\
&+\tag{f25}\label{f25}
\sum\limits_{\footnotesize\begin{subarray}{c}p+q=n+1\\ p,q>0\end{subarray}}( {\Id}_C \t \d_p )\c \d_q \c \psi_C.
\end{align}
By (\ref{def4}) and (\ref{def2.1}), for $m\leq n$, it follows that
\begin{align}\label{eq-f22'}
\sum_{k+l=m \atop k,l\geq0}\Delta_k \c \psi_l=\sum_{k+l=m \atop k,l\geq0}\big((\psi_l\otimes {\Id}_C)\circ \Delta_k + ({\Id}_C\otimes \psi_l)\circ \Delta_k\big).
\end{align}
\begin{align}
(\ref{f22})=&\sum_{p=1}^{n}\sum_{q+j=n+1-p \atop q,j\geq0}(\d_p \t {\Id}_C )\c \d_q \c \psi_j \notag \\
\stackrel{(\ref{eq-f22'})}{=}&\sum_{p=1}^{n}\sum_{q+j=n+1-p \atop q,j\geq0}(\d_p \t {\Id}_C )\c\big((\psi_j\otimes {\Id}_C)\circ \Delta_q + ({\Id}_C\otimes \psi_j)\circ \Delta_q \big) \notag \\
=&\tag{f26}\label{f26}
\sum_{p+j=n+1 \atop p,j>0}(\d_p \t {\Id}_C )\c\big((\psi_j\otimes {\Id}_C)\circ \Delta + ({\Id}_C\otimes \psi_j)\circ \Delta \big)\\
&+\tag{f27}\label{f27}
\sum_{p+q+j=n+1 \atop p,q>0,~j\geq0}(\d_p \t {\Id}_C )\c\big((\psi_j\otimes {\Id}_C)\circ \Delta_q + ({\Id}_C\otimes \psi_j)\circ \Delta_q \big)
\end{align}
and
\begin{align}
(\ref{f24})=&-\sum_{p=1}^{n}\sum_{q+j=n+1-p \atop q,j\geq0}( {\Id}_A \t \d_p )\c \d_q \c \psi_j \notag \\
\stackrel{(\ref{eq-f22'})}{=}&-\sum_{p=1}^{n}\sum_{q+j=n+1-p \atop q,j\geq0}( {\Id}_C \t \d_p )\c\big((\psi_j\otimes {\Id}_C)\circ \Delta_q + ({\Id}_C\otimes \psi_j)\circ \Delta_q \big) \notag \\
=-&\tag{f28}\label{f28}
\sum_{p+j=n+1 \atop p,j>0}( {\Id}_C \t \d_p )\c\big((\psi_j\otimes {\Id}_C)\circ \Delta + ({\Id}_C\otimes \psi_j)\circ \Delta \big)\\
-&\tag{f29}\label{f29}
\sum_{p+q+j=n+1 \atop p,q>0,~j\geq0}( {\Id}_C \t \d_p )\c\big((\psi_j\otimes {\Id}_C)\circ \Delta_q + ({\Id}_C\otimes \psi_j)\circ \Delta_q \big).
\end{align}
Now observe that each of the following equations:
\begin{align}
(\ref{f14})=&
\sum\limits_{\footnotesize\begin{subarray}{c}p+q+i=n+1\\ p\geq0,~q,i>0\end{subarray}}\Big[\big((\psi_q\otimes {\Id}_C)\circ \Delta_p + ({\Id}_C\otimes \psi_q)\circ \Delta_p \big)\t {\Id}_C\Big]\c \d_i \notag\\
=&\label{f141}\tag{f14$'$}
\sum\limits_{\footnotesize\begin{subarray}{c}p+q+i=n+1\\ p\geq0,~q,i>0\end{subarray}}(\psi_q \t {\Id}_C \t {\Id}_C)\c(\d_p \t {\Id}_C)\c \d_i \\
+&\label{f1411}\tag{f14$''$}
\sum\limits_{\footnotesize\begin{subarray}{c}p+q+i=n+1\\ p\geq0,~q,i>0\end{subarray}}( {\Id}_C \t \psi_q \t {\Id}_C)\c(\d_p \t {\Id}_C)\c \d_i.
\end{align}
\begin{align}
(\ref{f17})=&
-\sum\limits_{\footnotesize\begin{subarray}{c}p+q+i=n+1\\ p\geq0~i,q>0\end{subarray}}\Big[{\Id}_C \t\big((\psi_q\otimes {\Id}_C)\circ \Delta_p + ({\Id}_C\otimes \psi_q)\circ \Delta_p \big)\Big]\c \d_i  \notag\\
=&\label{f171}\tag{f17$'$}-\sum\limits_{\footnotesize\begin{subarray}{c}p+q+i=n+1\\ p\geq0,~q,i>0\end{subarray}}( {\Id}_C \t \psi_q \t {\Id}_C)\c( {\Id}_C \t \d_p)\c \d_i \\
&-\label{f1711}\tag{f17$''$}\sum\limits_{\footnotesize\begin{subarray}{c}p+q+i=n+1\\ p\geq0,~q,i>0\end{subarray}}( {\Id}_C \t  {\Id}_C \t \psi_q)\c( {\Id}_C \t \d_p)\c \d_i.
\end{align}
\begin{align}
(\ref{f26})=&\sum_{p+j=n+1 \atop p,j>0}(\d_p \t {\Id}_C )\c\big((\psi_j\otimes {\Id}_C)\circ \Delta + ({\Id}_C\otimes \psi_j)\circ \Delta \big) \notag\\
=&\label{f261}\tag{f26$'$}
\sum_{p+j=n+1 \atop p,j>0}(\d_p \c \psi_j \t {\Id}_C )\circ \Delta \\
&+\label{f2611}\tag{f26$''$}
\sum_{p+j=n+1 \atop p,j>0}({\Id}_C\otimes {\Id}_C\otimes \psi_j)\c(\d_p \t {\Id}_C) \c \d.
\end{align}
\begin{align}
&(\ref{f27})\notag\\
=&\sum_{p+q+j=n+1 \atop p,q,>0,~j\geq0}(\d_p \c \psi_j \t {\Id}_C)\c \d_q+\underline{\sum_{p+q+j=n+1 \atop p,q>0,~j\geq0}({\Id}_C \t {\Id}_C \t \psi_j )\c (\d_p\otimes {\Id}_C)\circ \Delta_q} \notag\\
=&\label{f271}\tag{f27$'$}
\underline{\sum_{p+q=n+1 \atop p,q>0}({\Id}_C \t {\Id}_C \t \psi_C )\c (\d_p\otimes {\Id}_C)\circ \Delta_q} \\
&+\label{f2711}\tag{f27$''$}
\underline{\sum_{p+q+j=n+1 \atop p,q,j>0}({\Id}_C \t {\Id}_C \t \psi_j )\c (\d_p\otimes {\Id}_C)\circ \Delta_q }\\
&+\label{f27111}\tag{f27$'''$}
\sum_{p+q+j=n+1 \atop p,q>0,~j\geq0}(\d_p \c \psi_j \t {\Id}_C)\c \d_q.
\end{align}
\begin{align}
&(\ref{f28})\notag\\
=&-\label{f281}\tag{f28$'$}
\sum_{p+j=n+1 \atop p,j>0}(\psi_j \t{\Id}_C \t {\Id}_C)\c({\Id}_C \t \d_p)\c \d \\
&-\label{f2811}\tag{f28$''$}
\sum_{p+j=n+1 \atop p,j>0}({\Id}_C \t \d_p \c \psi_j)\c \d.
\end{align}
\begin{align}
&(\ref{f29})\notag\\
=&-\sum_{p+q+j=n+1 \atop p,q>0,~j\geq0}(\psi_j \t{\Id}_C \t {\Id}_C)\c ({\Id}_C \t \d_p)\c \d_q -\sum_{p+q+j=n+1 \atop p,q>0,~j\geq0}
({\Id}_C \t \d_p\c \psi_j)\c \d_q \notag \\
=&\label{f291}\tag{f29$'$}
-\sum_{p+q=n+1 \atop p,q>0}(\psi_C \t{\Id}_C \t {\Id}_C)\c ({\Id}_C \t \d_p)\c \d_q \\
&-\label{f2911}\tag{f29$''$}
\sum_{p+q+j=n+1 \atop p,q,j>0}(\psi_j \t{\Id}_C \t {\Id}_C)\c ({\Id}_C \t \d_p)\c \d_q \\
&-\label{f29111}\tag{f29$'''$}
\sum_{p+q+j=n+1 \atop p,q>0,~j\geq0}
({\Id}_C \t \d_p\c \psi_j)\c \d_q.
\end{align}
We observe that each of the following sums is equal to 0: (\ref{e1}) + (\ref{f23}) = 0, (\ref{e2}) + (\ref{f25}) = 0, (\ref{f4}) + (\ref{f19}) = 0, (\ref{f18}) + (\ref{f29111}) = 0, (\ref{f1}) + (\ref{f2811}) = 0, (\ref{f10}) + (\ref{f261}) = 0, (\ref{f15}) + (\ref{f27111}) = 0.

\hspace{5em} (\ref{f13}) + (\ref{f16}) + (\ref{f291}) + (\ref{f271}) %写在公式环境里会出问题
\begin{align*}
=&\sum_{p+i=n+1 \atop p,i>0}(\psi_C\t {\Id}_C \t {\Id}_C)\c(\d_p \t {\Id}_C)\c \d_i \\
&+\sum_{p+i=n+1 \atop p,i>0}( {\Id}_C \t \psi_C\t {\Id}_C)\c(\d_p \t {\Id}_C)\c \d_i \\
&-\sum_{p+i=n+1 \atop p,i>0}( {\Id}_C \t \psi_C\t {\Id}_C)\c({\Id}_C \t \d_p)\c \d_i \\
&-\sum_{p+i=n+1 \atop p,i>0}( {\Id}_C \t{\Id}_C \t \psi_C)\c({\Id}_C \t \d_p)\c \d_i \\
&-\sum_{p+q=n+1 \atop p,q>0}(\psi_C \t{\Id}_C \t {\Id}_C)\c ({\Id}_C \t \d_p)\c \d_q \\
&+\sum_{p+q=n+1 \atop p,q>0}({\Id}_C \t {\Id}_C \t \psi_C )\c ((\d_p\otimes {\Id}_C)\circ \Delta_q
=-(\ref{e3}),
\end{align*}
i.e., (\ref{f13}) + (\ref{f16}) + (\ref{f291}) + (\ref{f271}) + (\ref{e3}) = 0.\\
Using (\ref{def3}) and (\ref{def1.1}) again, for $m\leq n$, we have
\begin{align}\label{eq-f2sum}
\sum_{k+l=m \atop k,l\geq 0}({\Id}_C \otimes \Delta_k) \c\Delta_l = \sum_{k+l=m \atop k,l\geq 0} (\Delta_k \otimes {\Id}_C) \c\Delta_l
\end{align}
So
\hspace{4em}(\ref{f2}) + (\ref{f1711}) + (\ref{f5}) + (\ref{f2611}) + (\ref{f2711})
\begin{align*}
=&-\sum_{q+p=n+1 \atop q,p>0}({\Id}_C \t {\Id}_C\t \psi_q)\c({\Id}_C \t \d_p)\c \d \\
&-\sum\limits_{\footnotesize\begin{subarray}{c}p+q+i=n+1\\ p\geq0,~q,i>0\end{subarray}}( {\Id}_C \t  {\Id}_C \t \psi_q)\c( {\Id}_C \t \d_p)\c \d_i \\
&+\sum_{i+q=n+1 \atop i,q>0}({\Id}_C\t {\Id}_C \t \psi_q )\c(\d \t {\Id}_C)\c \d_i \\
&+\sum_{p+q=n+1 \atop p,q>0}({\Id}_C\otimes {\Id}_C\otimes \psi_q)\c(\d_p \t {\Id}_C) \c \d \\
&+\sum_{p+q+i=n+1 \atop p,i,q>0}({\Id}_C\t {\Id}_C \t \psi_q )\c (\d_p \t {\Id}_C)\c\d_i \\
=&-\sum_{p+q+i=n+1 \atop i,p\geq0,q>0}({\Id}_C \t {\Id}_C \t \psi_q)\c ({\Id}_C \t \d_p)\c \d_i \\
&+\sum_{p+q+i=n+1 \atop i,p\geq0,q>0}({\Id}_C \t {\Id}_C \t \psi_q)\c ( \d_p \t {\Id}_C )\c \d_i \\
=&\sum_{q=1}^{n}\sum_{p+i=n+1-q \atop i,p\geq0}({\Id}_C \t {\Id}_C \t \psi_q)\c\big(( \d_p \t {\Id}_C )\c \d_i -  ({\Id}_C \t \d_p)\c \d_i  \big)\\
\stackrel{(\ref{eq-f2sum})}{=}&0.
\end{align*}
\hspace{4em}(\ref{f3}) + (\ref{f171}) + (\ref{f11}) + (\ref{f1411})
\begin{align*}
=&-
\sum_{q+p=n+1 \atop q,p>0}({\Id}_C \t \psi_q \t{\Id}_C )\c({\Id}_C \t \d_p)\c \d \\
&-\sum\limits_{\footnotesize\begin{subarray}{c}p+q+i=n+1\\ p\geq0,~q,i>0\end{subarray}}( {\Id}_C \t \psi_q \t {\Id}_C)\c( {\Id}_C \t \d_p)\c \d_i\\
&+
\sum_{q+p=n+1 \atop q,p>0}({\Id}_C \t \psi_q \t {\Id}_C)\c(\d_p \t {\Id}_C )\c \d \\
&+
\sum\limits_{\footnotesize\begin{subarray}{c}p+q+i=n+1\\ p\geq0,~q,i>0\end{subarray}}( {\Id}_C \t \psi_q \t {\Id}_C)\c(\d_p \t {\Id}_C)\c \d_i\\
=&-\sum_{p+q+i=n+1 \atop i,p\geq0,q>0}({\Id}_C \t \psi_q \t{\Id}_C )\c ({\Id}_C \t \d_p)\c \d_i \\
&+\sum_{p+q+i=n+1 \atop i,p\geq0,q>0}({\Id}_C \t \psi_q \t{\Id}_C )\c ( \d_p \t {\Id}_C )\c \d_i \\
=&\sum_{q=1}^{n}\sum_{p+i=n+1-q \atop i,p\geq0}({\Id}_C \t \psi_q \t{\Id}_C )\c\big(( \d_p \t {\Id}_C )\c \d_i -  ({\Id}_C \t \d_p)\c \d_i  \big)\\
\stackrel{(\ref{eq-f2sum})}{=}&0.
\end{align*}
\hspace{4em}(\ref{f9}) + (\ref{f281}) + (\ref{f2911}) + (\ref{f12}) + (\ref{f141})
\begin{align*}
=&-
\sum_{i+q=n+1 \atop i,q>0}(\psi_q \t {\Id}_C \t {\Id}_C)\c ( {\Id}_C \t \d)\c\d_i \\
&-
\sum_{p+q=n+1 \atop p,q>0}(\psi_q \t{\Id}_C\t {\Id}_C)\c({\Id}_C \t \d_p)\c \d \\
&-
\sum_{p+q+i=n+1 \atop p,q,i>0}(\psi_q \t{\Id}_C \t {\Id}_C)\c ({\Id}_C \t \d_p)\c \d_i \\
&+
\sum_{p+q=n+1 \atop p,q>0}(\psi_q \t {\Id}_C \t {\Id}_C)\c(\d_p \t {\Id}_C )\c \d \\
&+
\sum\limits_{\footnotesize\begin{subarray}{c}p+q+i=n+1\\ p\geq0,~q,i>0\end{subarray}}(\psi_q \t {\Id}_C \t {\Id}_C)\c(\d_p \t {\Id}_C)\c \d_i \\
=&-\sum_{p+q+i=n+1 \atop i,p\geq0,q>0}( \psi_q \t{\Id}_C \t{\Id}_C)\c ({\Id}_C \t \d_p)\c \d_i \\
&+\sum_{p+q+i=n+1 \atop i,p\geq0,q>0}( \psi_q \t{\Id}_C \t{\Id}_C )\c ( \d_p \t {\Id}_C )\c \d_i \\
=&\sum_{q=1}^{n}\sum_{p+i=n+1-q \atop i,p\geq0}( \psi_q \t{\Id}_C \t{\Id}_C )\c\big(( \d_p \t {\Id}_C )\c \d_i -  ({{\Id}}_C \t \d_p)\c \d_i  \big)\\
\stackrel{(\ref{eq-f2sum})}{=}&0.
\end{align*}

So we have $\delta_c  ({\Ob}_{\psi}^{2})-\omega ({\Ob}_{C}^{3})=0$, which implies that ${\Ob}_\Omega$ is a $3$-cocycle.

The proof of Proposition \ref{ext2} is finished.

\hspace*{1cm}\\

\noindent{\bf Acknowledge} The authors would like to thank the anonymous referee for his/her helpful comments. In particular, the interesting question in Remark \ref{rmk Ore}  were raised by him/her.

\end{document}